\documentclass[11pt]{article}
\usepackage{graphicx}
\usepackage{color}
\usepackage{amsmath}
\usepackage{amssymb}
\usepackage{amscd}
\usepackage{bbm}
\usepackage{tcolorbox}
\usepackage{fancyhdr,a4wide}
\usepackage{amsthm}

\newcommand{\R}{\mathbb{R}}
\newcommand{\inr}[1]{\left\langle #1 \right\rangle}

\newcommand{\E}{\mathbb{E}}

\newcommand{\PP}{\mathbb{P}}

\newcommand{\eps}{\varepsilon}

\newtheorem{Theorem}{Theorem}[section]
\newtheorem{Lemma}[Theorem]{Lemma}

\newtheorem{Proposition}[Theorem]{Proposition}
\newtheorem{Corollary}[Theorem]{Corollary}

\theoremstyle{definition}
\newtheorem{Definition}[Theorem]{Definition}
\newtheorem{Remark}[Theorem]{Remark}

\newtheorem{Question}[Theorem]{Question}

\newtheorem{Example}[Theorem]{Example}

\numberwithin{equation}{section}

\title{Structure preservation via the Wasserstein distance}
\author{
Daniel Bartl\footnote{
University of Vienna, Faculty of Mathematics  (daniel.bartl@univie.ac.at)}
 \and
Shahar Mendelson\footnote{ETH Z\"urich, Department of Mathematics 
(shahar.mendelson@gmail.com)}
}

\begin{document}
\maketitle
\begin{abstract}
We show that under minimal assumptions on a random vector $X\in\mathbb{R}^d$ and with high probability, given $m$ independent copies of $X$, the coordinate distribution of each vector $(\langle X_i,\theta \rangle)_{i=1}^m$ is dictated by the distribution of the true marginal $\langle X,\theta \rangle$. 
Specifically, we show that with high probability,
\[\sup_{\theta \in S^{d-1}} \left( \frac{1}{m}\sum_{i=1}^m \left|\langle X_i,\theta \rangle^\sharp - \lambda^\theta_i \right|^2 \right)^{1/2}
\leq c \left( \frac{d}{m}  \right)^{1/4},\]
where  $\lambda^{\theta}_i = m\int_{(\frac{i-1}{m}, \frac{i}{m}]}  F_{ \langle X,\theta \rangle }^{-1}(u)\,du$ and $a^\sharp$ denotes the monotone non-decreasing rearrangement of $a$.
Moreover, this estimate is optimal.

The proof follows from a sharp estimate on the worst Wasserstein distance between a marginal of $X$ and its empirical counterpart, $\frac{1}{m} \sum_{i=1}^m \delta_{\langle X_i, \theta \rangle}$.
\end{abstract}

\section{Introduction} \label{sec:intro}

The study of the way in which structure can be preserved using random sampling is of central importance in modern mathematics. Various notions of structure have been considered over the years, resulting in numerous applications in pure mathematics, statistics and data science.

Let $\mu$ be a centred probability measure on $\R^d$ and let $X$ be distributed according to $\mu$.
We focus on the way in which a typical sample, consisting of $m$ independent copies of $X$, inherits features of that measure.
Informally put, we explore the following:

\begin{tcolorbox}
\begin{Question} \label{Qu:infromal}
Given $X_1,\dots,X_m$ selected independently according to $\mu$, how much ``information" on $\mu$ can be extracted (with high probability) from the set $\{X_1,\dots,X_m\}$?
\end{Question}
\end{tcolorbox}

Since the empirical measure $\mu_m=\frac{1}{m}\sum_{i=1}^m \delta_{X_i}$ converges weakly to $\mu$ almost surely, the main interest in Question \ref{Qu:infromal} is of a quantitative nature.

\vspace{0.5em}
To put this (still rather vague) question in some perspective, let us start with a natural notion that will prove to be instructive, but at the same time rather useless. Consider two independent samples $(X_i)_{i=1}^m$ and $(X_i^\prime)_{i=1}^m$, both selected according to $\mu^{\otimes m}$. Intuitively, if for a typical sample $(X_i)_{i=1}^m$, the set $\{X_i : 1 \leq i \leq m\}$ inherits much of $\mu$'s structure, then $\{X_i : 1 \leq i \leq m\}$ and $\{X_i^\prime : 1 \leq i \leq m\}$ should be ``close" to each other. And, because the way the points are ordered is irrelevant, a natural notion of similarity between the two ``clouds" of points is
\begin{align}
\label{eq:Wasserstein.could}
\inf_\pi \left( \frac{1}{m}\sum_{i=1}^m \left\| X_{\pi(i)} -X_i' \right\|_2^2 \right)^{1/2},
\end{align}
with the infimum taken over all permutations $\pi$ of $\{1,\dots,m\}$.

As it happens, \eqref{eq:Wasserstein.could} is simply the $\mathcal{W}_2$  \emph{Wasserstein distance} between the two empirical measures $\frac{1}{m}\sum_{i=1}^m \delta_{X_i}$ and $\frac{1}{m}\sum_{i=1}^m \delta_{X_i'}$. The $\mathcal{W}_2$ distance is defined on $\mathcal{P}_2(\R^d)$---the set of Borel probability measures on $\R^d$ with finite second moment---by
\[
\mathcal{W}_2(\tau,\nu)=\inf_\Pi \left( \int_{\mathbb{R}^d\times\mathbb{R}^d } \|x-y\|_2^2 \,\Pi(dx,dy) \right)^{1/2}.
\]
Here the infimum taken over all \emph{couplings} $\Pi$, that is, probability measures whose first marginal is $\tau$ and their second marginal is $\nu$.
For detailed surveys on the Wasserstein distance, see, e.g., \cite{figalli2021invitation,villani2021topics}.

Setting $\mu_m=\frac{1}{m}\sum_{i=1}^m\delta_{X_i}$ and $\mu_m^\prime =\frac{1}{m}\sum_{i=1}^m \delta_{X_i^\prime}$, it is straightforward to verify that \eqref{eq:Wasserstein.could} is simply $\mathcal{W}_2(\mu_m,\mu_m^\prime)$. Moreover, a standard convexity argument shows that obtaining high probability estimates on $\mathcal{W}_2(\mu_m,\mu_m')$ and on $\mathcal{W}_2(\mu_m,\mu)$ are equivalent questions, and in what follows we focus on the latter.

While $\mathcal{W}_2(\mu_m,\mu)$ is a natural way of comparing $\mu_m$ and $\mu$, using the Wasserstein distance has a significant drawback. In the high-dimensional setup, $\mathcal{W}_2$ is just \emph{too sensitive}: the typical distance between $\mu_m$ and $\mu$ is almost diametric unless $m$ is exponential in the dimension $d$.
Indeed, although $\mathcal{W}_2(\mu_m,\mu) \to 0$ almost surely as $m\to\infty$,  the pointwise best approximation of $\mu$ by $m$ points, i.e., $\inf_{x_1,\dots,x_m\in \R^d} \mathcal{W}_2(\frac{1}{m}\sum_{i=1}^m\delta_{x_i},\mu)$, typically scales like $m^{-1/d}$ even for well-behaved measures---like the gaussian measure or the uniform measure on the unit cube, see \cite{dudley1969speed}.

The slow decay of $\mathcal{W}_2(\mu_m,\mu)$ is a manifestation of the \emph{curse of dimensionality} phenomenon, rendering that notion of similarity useless for our purposes. It also hints that if one is to overcome the curse of dimensionality, a useful notion of structure preservation should depend only on low-dimensional marginals of $\mu$.

One such  notion is based on the behaviour of the \emph{extremal singular values} of the random matrix whose rows are $X_1,...,X_m$. 
Assume for the sake of simplicity that the centred measure $\mu$ is also isotropic (that is, its covariance is the identity) and consider $S^{d-1}$, the Euclidean unit sphere in $\R^d$, viewed as a class of linear functionals $\{ \inr{\theta,\cdot} : \theta \in S^{d-1}\} \subset L_2(\mu)$. Let $X$ be distributed according to $\mu$, set $X_1,\dots,X_m$ to be independent copies of $X$ and put
\[
\Gamma = \frac{1}{\sqrt{m}} \sum_{i=1}^m \inr{X_i,\cdot}e_i.
\]
The random operator $\Gamma$ is an embedding of $(\R^d,\|\cdot\|_{L_2(\mu)})$ into $\ell_2^m=(\R^m,\|\cdot\|_2)$, and it is natural to identify conditions on $X$ and $m$ under which $\Gamma$ is a $1 \pm \eps$ isomorphism, i.e., that
\begin{align} \label{eq:intro.BY}
\sup_{\theta \in S^{d-1}} \left| \|\Gamma \theta\|_2^2 - 1 \right| \leq \eps.
\end{align}
Clearly, \eqref{eq:intro.BY} means that the extremal singular values of $\Gamma$ are close to $1$, and the extent to which the random sample $X_1,\dots,X_m$ inherits the $L_2$ structure endowed by $\mu$ is measured by \eqref{eq:intro.BY}.

While \eqref{eq:intro.BY} captures significant information on structure preservation, it is still rather crude: it yields very little information on the measure $\mu$---other than exhibiting that it is isotropic. And, if $\mu$ is not isotropic, analogs of \eqref{eq:intro.BY} allow one to recover $\mu$'s covariance structure (see, for example, \cite{lugosi2020multivariate,mendelson2020approximating,mendelson2020robust}  for results of that flavour), but nothing beyond that. In comparison, far more accurate information is `coded' in estimates on the `worst' Wasserstein distance of a one-dimensional marginal of $\mu$ from its empirical counterpart. 
To see why, consider a direction $\theta\in S^{d-1}$, let $\mu^\theta(A)=\mu(\{x\in \R^d : \inr{x,\theta}\in A\})$ be the marginal of $\mu$ in direction $\theta$ and set \[F_{\mu^\theta}(t)=\mu^\theta( (-\infty,t] )
\quad\text{and}\quad
F_{\mu^\theta}^{-1}(u)=\inf\left\{ t\in\R : F_{\mu^\theta}(t)\geq u \right\}\]
to be the \emph{distribution function} of $\mu^\theta$ and its \emph{(right-)inverse}.

\begin{Definition}
\label{def:msW}
For every $\mu,\nu\in\mathcal{P}_2(\R^d)$, the \emph{max-sliced Wasserstein distance} is defined by
\[ \mathcal{SW}_2(\mu,\nu)=\sup_{\theta\in S^{d-1}} \mathcal{W}_2\left(\mu^\theta,\nu^\theta\right). \]
\end{Definition}

One can show that $\mathcal{SW}_2$ is a metric on $\mathcal{P}_2(\mathbb{R}^d)$ that endows the same topology as $\mathcal{W}_2$.
For the proof of that fact and more information on similar notions, see, e.g.,  \cite{deshpande2019max,lin2021projection,manole2019minimax,nietert2022statistical,olea2022generalization,paty2019subspace}.

\vspace{0.5em}
An observation that is used frequently in what follows is a closed-form of the Wasserstein distance between one-dimensional measures:
\begin{equation} 
\label{eq:co-monotone}
\mathcal{SW}_2(\mu_m,\mu)
=\sup_{\theta\in S^{d-1}} \left( \int_0^1 \left( F_{\mu^\theta_m}^{-1}(u) - F_{\mu^\theta}^{-1}(u) \right)^2 \,du \right)^{1/2},
\end{equation}
see, for example, \cite{ruschendorf1985wasserstein} and  Lemma \ref{lem:Wasserstein.via.inverse.functions}.

Using the representation \eqref{eq:co-monotone}, the following is of central importance:
\begin{tcolorbox}
An upper bound on  $\mathcal{SW}_2(\mu_m,\mu)$ implies uniform concentration of the coordinate distribution of the vectors $(\inr{X_i,\theta})_{i=1}^m$ around a well-determined set of values, endowed by $\mu$.
\end{tcolorbox}

To be more precise, let $\theta\in S^{d-1}$ and
for $1\leq i\leq m$ set
\[
\lambda^\theta_i= m\int_{(i-1)/m}^{i/m} F_{\mu^\theta}^{-1}(u) \,du.
\]
One may show (see Section \ref{sec:preliminary} for the proof) that
\begin{align} \label{eq:coordinate.distribution}
\sup_{\theta\in S^{d-1}} \left( \frac{1}{m}\sum_{i=1}^m \left| \inr{X_i,\theta}^\sharp - \lambda^\theta_i \right|^2 \right)^{1/2} \leq \mathcal{SW}_2(\mu_m,\mu),
\end{align}
where here and throughout this article,  $(a_i^\sharp)_{i=1}^m$ is non-decreasing rearrangement of $(a_i)_{i=1}^m$.
Equation \eqref{eq:coordinate.distribution}  means that if $\mathcal{SW}_2(\mu_m,\mu)$ is `small', the random vectors $(\inr{X_i,\theta})_{i=1}^m$ all inherit---on the same event---the distribution of the corresponding one-dimensional marginals of $\mu$.

\vspace{0.5em}
\begin{tcolorbox}
Roughly put, our main result is that
 under minimal assumptions on $X$ and with high probability, $\mathcal{SW}_2(\mu_m,\mu)$ is indeed small:
\begin{equation} \label{eq:optimal-est-into-1}
\mathcal{SW}_2(\mu_m,\mu)
\leq c\left( \frac{d}{m}  \right)^{1/4}  ,
\end{equation}
where $c$ is a suitable constant---and this estimate is optimal.

In particular, using the above notation, we have that
\begin{equation} \label{eq:optimal-est-into-2}
\sup_{\theta \in S^{d-1}} \left( \frac{1}{m}\sum_{i=1}^m \left| \inr{X_i,\theta}^\sharp - \lambda^\theta_i \right|^2 \right)^{1/2}
\leq c \left( \frac{d}{m}  \right)^{1/4}.
\end{equation}
\end{tcolorbox}

\begin{Remark}
Set $q^\theta_i = F_{\mu^\theta}^{-1}\left(\frac{i}{m}\right)$ for $i < m$ and $q^\theta_m = q^\theta_{m-1}$.
Inequality \eqref{eq:optimal-est-into-2} remains valid if the averaged quantiles $(\lambda^\theta_i)_{i=1}^m$ are replaced by the quantiles $(q^\theta_i)_{i=1}^m$---see Lemma \ref{lem:Gamma.theta.quantils.averaged.vs.not} for the details.
	Moreover, we show in Lemma \ref{lem:Gamma.theta.geq.SW2} that under minimal assumptions on $X$,
\[
	\sup_{\theta\in S^{d-1}} \left( \frac{1}{m}\sum_{i=1}^m \left| \inr{X_i,\theta}^\sharp - \lambda^\theta_i \right|^2 \right)^{1/2} 
	\geq \mathcal{SW}_2(\mu_m,\mu) - c \left(\frac{1}{m}\right)^{1/4}.
\]
	Thus, inequality \eqref{eq:coordinate.distribution} can be almost reversed, and
	the optimality of \eqref{eq:optimal-est-into-1} ensures that \eqref{eq:optimal-est-into-2} is also optimal.
\end{Remark}

To explain what is meant by ``minimal assumptions on $X$", let us compare \eqref{eq:optimal-est-into-1} with \eqref{eq:intro.BY}.
Both  estimates offer some form of uniform control on the one-dimensional marginals of $X$, with one (trivially) stronger than the other. 
Indeed, set
\begin{align} \label{eq:def.rho.d.m}
\rho_{d,m}
&=
\sup_{\theta\in S^{d-1}} \left|  \frac{1}{m}\sum_{i=1}^m \inr{X_i,\theta}^2  -1 \right|
=\sup_{\theta\in S^{d-1}} \left| \|\Gamma\theta\|_2^2-1 \right|
\end{align}
and recall that by isotropicity $\E\inr{X,\theta}^2=1$.
In particular,
\[ \rho_{d,m}=\sup_{\theta\in S^{d-1}} \left| \int_0^1 ( F_{\mu^\theta_m}^{-1}(u))^2\,du - \int_0^1 ( F_{\mu^\theta}^{-1}(u))^2 \,du \right| ;\]
thus, $\rho_{d,m}$ is the maximal difference between the second moments of the true and empirical (inverse) distributions of marginals $\| F_{\mu^\theta_m}^{-1}\|_{L_2}^2$ and $\| F_{\mu^\theta}^{-1}\|_{L_2}^2$.
In contrast, following \eqref{eq:co-monotone}, $\mathcal{SW}_2(\mu_m,\mu)$ is the maximal $L_2$-distance $\|  F_{\mu^\theta_m}^{-1}- F_{\mu^\theta}^{-1}\|_{L_2}$, and it is straightforward to verify that if $\mathcal{SW}_2( \mu_m,\mu)\leq 1$, then
	\begin{align}
	\label{eq:Bai.Yin}
	\rho_{d,m}
	\leq 3\cdot \mathcal{SW}_2( \mu_m,\mu )
	\end{align}
(see Lemma \ref{lem:Wasserstein.implies.Bai.Yin} for the proof).

A priori, there is no reason why \eqref{eq:Bai.Yin} should be anything but a crude upper estimate: expecting that  $F_{\mu^\theta_m}^{-1}$ and $F_{\mu^\theta}^{-1}$ are close in $L_2$  solely because their second moments are similar seems unrealistically optimistic.
Surprisingly, our main result  shows that \eqref{eq:Bai.Yin} can be reversed  to some extent, and $\mathcal{SW}_2( \mu_m,\mu)$ can be controlled in terms of $\rho_{d,m}$:

\begin{tcolorbox}
\begin{Theorem}
\label{thm:Wasserstein.small}
	Let $X$ be centred and isotropic, and assume that $\sup_{\theta \in S^{d-1}} \|\inr{X,\theta}\|_{L_q} \leq L$ for some $q\geq 4$.
Then there are absolute constants $c_0,c_1,c_2$ and a constant $c_3$ that depends only on $q$ and $L$ such that the following holds.

	Let  $0<\Delta\leq c_0$ and set $m \geq c_1 \frac{d}{\Delta}$.
	With probability at least $1-\exp(-c_2 \Delta m)$,
	\[  \mathcal{SW}_2(\mu_m,\mu)
	\leq c_3
	\begin{cases}
	 \rho_{d,m}^{1/2}+\Delta^{1/4}  &\text{if } q>4,\\
\\
	\rho_{d,m}^{1/2}+\Delta^{1/4}\log\left(\frac{1}{\Delta}\right)  &\text{if } q=4.
	\end{cases} \]
\end{Theorem}
\end{tcolorbox}

\subsubsection*{On the optimality of Theorem \ref{thm:Wasserstein.small}}

The question of the optimality of Theorem \ref{thm:Wasserstein.small} is explored in Section \ref{sec:lower-W2}. 
For example, we show that the estimate of $(d/m)^{1/4}$  is sharp in general---even if one allows for a significantly weaker probability estimate---of constant probability rather than of $1-\exp(-cd)$.
Moreover, the best estimate that one can hope to get with  probability at least $1-\exp(-c\Delta m)$ is $\Delta^{1/4}$,  even if the random vector $X$ is well-behaved, e.g., isotropic and $L$-subgaussian\footnote{Recall that if a random vector $X$ is isotropic and $L$-subgaussian then $\|\inr{X,\theta}\|_{L_p} \leq L \sqrt{p}$ for every $p\geq 2$ and $\theta\in S^{d-1}$.}. 
The latter stems from a simple one-dimensional phenomenon: it is straightforward to verify that for $\Delta\leq 1/4$ and $m\geq 4$, if $\mu^\theta$ is symmetric and $\{-1,1\}$-valued, then
	\begin{align}
	\label{eq:int.Delta.lower.bound.bernoulli}
	\int_{1/4}^{3/4}  \left( F_{\mu^{\theta}_m}^{-1}(u) - F_{\mu^{\theta}}^{-1}(u)\right)^2 \,du
	\geq c_1\sqrt\Delta
	\end{align}
with probability at least $c_2\exp(-c_3\Delta m)$ (see Lemma \ref{lem:Delta.1.4} for the proof). Thus, invoking the representation \eqref{eq:co-monotone} it is evident that if $X$ is uniformly distributed in  $\{-1,1\}^d$, then with probability at least $c_2\exp(-c_3\Delta m)$,
	\begin{align}
	\label{eq:Delta.1.4.sharp}
	\mathcal{SW}_2(\mu_m,\mu)
	\geq \sqrt{c_1} \Delta^{1/4};
	\end{align}
	and the fact that $X$ is subgaussian does not mean that the error estimate in Theorem \ref{thm:Wasserstein.small} can be better than $\sim \Delta^{1/4}$.

At the same time, once the random vector $X$ is `regular' in a certain sense, the estimate in Theorem \ref{thm:Wasserstein.small} can be improved to $\sqrt \Delta$ (up to logarithmic factors)---in which case  \eqref{eq:Bai.Yin} truly can be reversed---see Theorem \ref{cor:Wasserstein.small.log-concave.subgaussian}.
This phenomenon is in line with the results from Bobkov-Ledoux \cite{bobkov2019one}, which are focused on the behaviour of $\E \mathcal{W}_2(\mu_m,\mu)$ in the one-dimensional framework.
	In particular Bobkov and Ledoux explore conditions under which the error estimate is  $m^{-1/2}$  (up to logarithmic factors) and show that this can only happen under strong regularity assumptions on $F_\mu^{-1}$.
	Since the focus of this article is on the effect of high dimensions in the context of \emph{general} random vectors, we chose not pursue a similar path here.
	We only consider one family of regular random vectors---isotropic, log-concave random vectors---see Theorem \ref{cor:Wasserstein.small.log-concave.subgaussian}.

\subsubsection*{Why Theorem \ref{thm:Wasserstein.small} is surprising}

Theorem \ref{thm:Wasserstein.small}  implies a significant improvement on the state-of-the-art geometric estimates related to the study of the extremal singular values of the random matrix $\Gamma$. 
Sharp upper bounds on $\rho_{d,m}$ follow from (crude) structural information on the vectors $\{\Gamma \theta : \theta \in S^{d-1}\}$ (see, e.g., \cite{mendelson2014singular,tikhomirov2018sample}), and are based on (necessary) assumptions on the isotropic random vector $X$---namely that $\|X\|_2^2/d$ has a well-behaved tail-decay and that $\sup_{\theta \in S^{d-1}} \|\inr{X,\theta}\|_{L_q} \leq L$ for some $q \geq 4$.
At the heart of the arguments in \cite{mendelson2014singular,tikhomirov2018sample}  is that under those assumptions the following holds on a high probability event:
\begin{description}
\item{$\bullet$}
For each vector $w=\sqrt{m} \Gamma \theta=(\inr{X_i,\theta})_{i=1}^m$, the contribution of its largest $\Delta m$ coordinates to $\|w\|_2$  is not ``too big".
\item{$\bullet$}
The remaining coordinates ``live" within a multiplicative envelope, dictated by the assumed tail-decay.
For example, under an $L_q-L_2$ norm equivalence assumption one has that for every such vector $w$ and $k=\Delta m,\dots, m$
\begin{equation} \label{eq:multi-envelope}
w^\ast_k \leq c(L)(m/k)^{1/q},
\end{equation}
where $w^\ast$ denotes the monotone non-increasing rearrangement of $(|w_i|)_{i=1}^m$.
\end{description}

When $q>4$, this structural information suffices for establishing the optimal bound---that $\rho_{d,m} \leq c \sqrt{d/m}$ with high probability.

Although the structural information in \eqref{eq:multi-envelope} plays an instrumental role in obtaining sharp estimates on $\rho_{d,m}$, it is still significantly weaker than the outcome of Theorem \ref{thm:Wasserstein.small}.
Indeed, if $X$ satisfies a non-asymptotic `Bai-Yin' estimate, i.e., $\rho_{d,m}\leq c \sqrt\Delta$  with probability at least $1-\eta$, then it  follows from Theorem \ref{thm:Wasserstein.small} that with probability at least $1-\exp(-c_2 \Delta m)-\eta$,
\[  \mathcal{SW}_2(\mu_m,\mu)
	\leq
	\begin{cases}
	c_4 \Delta^{1/4} &\text{if } q>4,\\
\\
	c_4 \Delta^{1/4}\log\left(\frac{1}{\Delta}\right) &\text{if } q=4.
	\end{cases}  \]
Thus, the monotone non-decreasing rearrangement of each $(\inr{X_i,\theta})_{i=1}^m$ satisfies that
\[
\sup_{\theta \in S^{d-1}} \left( \frac{1}{m}\sum_{i=1}^m \left| \inr{X_i,\theta}^\sharp - \lambda^\theta_i \right|^2 \right)^{1/2}
\leq
	\begin{cases}
	c_5 \Delta^{1/4} &\text{if } q>4,\\
\\
	c_5 \Delta^{1/4}\log\left(\frac{1}{\Delta}\right) &\text{if } q=4,
	\end{cases}
\]
which is far more accurate than the existence of a one-sided multiplicative envelope function for the non-increasing rearrangement of vectors $(|\inr{X_i,\theta}|)_{i=1}^m$ as in \eqref{eq:multi-envelope}.

\begin{Remark}
When $q=4$, the state-of-the-art estimate on $\rho_{d,m}$ due to Tikhomirov \cite{tikhomirov2018sample} has an additional multiplicative logarithmic factor in $m/d$. And in a similar manner, the estimate in Theorem \ref{thm:Wasserstein.small} when $q=4$ also has an additional factor of $\log\frac{1}{\Delta}$ when $\Delta\geq d/m$.
The Bai-Yin asymptotics \cite{bai1993limit} imply that when $q=4$ the logarithmic factor in $\rho_{d,m}$ should disappear when $d,m\to\infty$ while keeping the ratio $d/m$ fixed. 
It is reasonable to expect that the same is true in regard to Theorem \ref{thm:Wasserstein.small} and that the $\log\frac{1}{\Delta}$ factor is superfluous.
\end{Remark}

\begin{Remark}
One scenario of particular interest is when the random vector $X$ is rotation invariant. In such a situation, an upper estimate on $\mathcal{SW}_2(\mu_m,\mu)$ combined with \eqref{eq:coordinate.distribution} implies that $\Gamma S^{d-1}$ inherits $\mu$'s invariance, in the sense that
\[ \sup_{\theta,\theta'\in S^{d-1}} \left( \frac{1}{m}\sum_{i=1}^m \left| \inr{X_i,\theta}^\sharp - \inr{X_i,\theta'}^\sharp  \right|^2 \right)^{1/2}
\leq 2\mathcal{SW}_2(\mu_m,\mu).\]
In other words, after replacing  vectors in $\Gamma S^{d-1}$ by their monotone rearrangement, the Euclidean diameter of the resulting set is bounded by $2\mathcal{SW}_2(\mu_m,\mu)$.
In particular, on the event in which \eqref{eq:optimal-est-into-1} holds,
\begin{equation} \label{eq:optimal-est-into-3}
\sup_{\theta,\theta^\prime \in S^{d-1}} \left( \frac{1}{m}\sum_{i=1}^m \left| \inr{X_i,\theta}^\sharp - \inr{X_i,\theta^\prime}^\sharp \right|^2 \right)^{1/2}
\leq 2c \left( \frac{d}{m}  \right)^{1/4} .
\end{equation}
The estimate in \eqref{eq:optimal-est-into-3} offers a way of proving that certain random functional are almost constant on the sphere and plays an instrumental role in the construction of a  non-gaussian Dvoretzky-Milman embedding in \cite{bartl2022nongaussian}.
\end{Remark}

We end this section with some comments on the behaviour of the (max-sliced) Wasserstein distance between the true measure and the empirical one when the true measure need not be isotropic. 

Consider $Y=\Sigma^{1/2}X$ where 
$X$ is a centred and isotropic random vector, and $\Sigma$ is a symmetric, positive-definite matrix. Denote by $\nu$ the measure endowed by $Y$, let  $\nu^\theta$ be its marginal distribution in the direction $\theta\in S^{d-1}$, and set $\sigma^2(\theta)= \mathbb{V}{\rm ar}[\langle Y,\theta\rangle]$.
The empirical measure of $\nu$ is denoted by $\nu_m$, and its marginal distribution in the direction $\theta$ is denoted by $\nu_m^\theta$.

Theorem \ref{thm:Wasserstein.small} immediately implies direction-dependent bounds on  $\mathcal{W}_2(\nu_m^\theta,\nu^\theta)$.
Indeed, if $X,\Delta$ and $m$ satisfy the assumptions of Theorem \ref{thm:Wasserstein.small}, then with probability at least $1-\exp(-c\Delta m)$,  for every $\theta\in S^{d-1}$, 
\[
 \mathcal{W}_2(\nu_m^\theta,\nu^\theta)
\leq  C
\begin{cases}
 \sigma(\theta) (\rho_{d,m}^{1/2} + \Delta^{1/4}) &\text{if } q>4, 
  \\
  \\
  \sigma(\theta) (\rho_{d,m}^{1/2} + \Delta^{1/4}\log(\frac{1}{\Delta})) &\text{if } q=4.
 \end{cases}
\]
Indeed, this  follows from the positive homogeneity of the Wasserstein distance $$
\mathcal{W}_2(\nu_m^\theta,\nu^\theta) = \sigma(\theta) \mathcal{W}_2(\mu_m^\theta,\mu^\theta) \leq \sigma(\theta)\mathcal{SW}_2(\mu_m,\mu),
$$ 
and the estimate on $\mathcal{SW}_2(\mu_m,\mu)$ follows from Theorem \ref{thm:Wasserstein.small}.

The behaviour of  $\mathcal{SW}_2(\nu_m,\nu)$ was recently studied in \cite{boedihardjo2024sharp}.
It was shown there that if $Y$ is symmetric and satisfies $\|Y\|_2\leq r$ almost surely for some $r\geq 1$, then for $\widetilde{\nu}_m=\frac{1}{2m}\sum_{i=1}^m (\delta_{Y_i}+ \delta_{-Y_i})$, and denoting by $\|\Sigma\|_{\rm op}$  the operator norm of $\Sigma$,
\begin{align}
\label{eq:SW.2.March}
 \mathbb{E}\left[ \mathcal{SW}_2\left( \widetilde{\nu}_m, \nu \right) \right]
\leq c \left( \left( \frac{ r^2 \log(em)}{m} \right)^{1/2} +  \left( \frac{ \|\Sigma\|_{\rm op} r^2  \log(em)}{m} \right)^{1/4} \right).
\end{align}
To put \eqref{eq:SW.2.March} in context,  an obvious lower bound on $r$ holds because $r^2 \geq \mathbb{E}[\|Y\|_2^2]= {\rm trace}(\Sigma)$.
Moreover, there are many natural situations in which $r$ must be significantly bigger than that trivial lower bound.
As a result, in the isotropic case ($\Sigma={\rm Id}$), $r^2\geq d$, and \eqref{eq:SW.2.March} is at least $(\frac{d\log(em)}{m})^{1/4}$, which is off by a logarithmic factor from our estimate of $(\frac{d}{m})^{1/4}$

Finally, while our focus is on $\mathcal{SW}_2$, the behaviour  of $\sup_{\theta\in \Theta}\mathcal{W}_2(\mu^\theta_m,\mu^\theta)$ for an arbitrary subset $\Theta \subset S^{d-1}$ was analysed recently in  \cite{bartl2024empirical} when $X$ is the standard gaussian random vector. 
Unlike the case $\Theta=S^{d-1}$, sharp bounds on $\sup_{\theta\in \Theta}\mathcal{W}_2(\mu^\theta_m,\mu^\theta)$ for  $\Theta \subset S^{d-1}$  are not known.

\subsection{A Warm up exercise: Theorem \ref{thm:Wasserstein.small} in the log-concave case}
\label{sec:warm.up}

The proof of Theorem \ref{thm:Wasserstein.small} is rather involved.
It is useful to present some of the ideas used in the argument in a simpler context---when $X$ is a centred, isotropic, log-concave random vector\footnote{Recall that a random vector is log-concave if it has a density that is a log-concave function.}.
As it happens, for such random vectors the estimate in Theorem \ref{thm:Wasserstein.small} can be improved from $\Delta^{1/4}$ to $\sqrt \Delta$ (up to logarithmic factors). 
At the same time, it is important to keep in mind that  $\Delta^{1/4}$ is the best one can hope for unless $X$ is regular in a strong sense---see Section \ref{sec:lower-W2}.

\begin{Theorem}
\label{cor:Wasserstein.small.log-concave.subgaussian}
There are absolute constants $c_0,c_1,c_2$ and $c_3$ for which the following holds.
Let $X$ be a centred, isotropic and log-concave random vector in $\R^d$.
Let $0<\Delta\leq c_0$ and set $m\geq c_1 \frac{d}{\Delta}$.
Then with probability at least $1-\exp(-c_2 \sqrt{\Delta m} \log^2(\frac{1}{\Delta}))$,
	\[
\mathcal{SW}_2( \mu_m,\mu )
	\leq c_3  \sqrt\Delta  \log^2\left( \frac{1}{\Delta} \right).
\]
If $X$ is, in addition, $L$-subgaussian, then the constants $c_2,c_3$ depend on $L$, and with probability at least $1-\exp(-c_2 \Delta m)$,
	\[
\mathcal{SW}_2( \mu_m,\mu )
	\leq c_3  \sqrt{\Delta}  \log^{3/2}\left( \frac{1}{\Delta} \right) .
\]
\end{Theorem}

\begin{Remark}
In the setting of Theorem \ref{cor:Wasserstein.small.log-concave.subgaussian}, if $\Delta = c_1d/m$ then with probability at least $1-\exp(-c^\prime_2 \sqrt{d}\log^2(\frac{m}{d}))$,
\begin{equation} \label{eq:lower-log-concave}
\mathcal{SW}_2( \mu_m,\mu)
	\leq c^{\prime}_3  \sqrt{\frac{d}{m}}  \log^2\left( \frac{m}{d} \right).
\end{equation}

We show in Section \ref{sec:lower-W2} that if $X$ is isotropic and is not degenerate, namely $\E \|X\|_2 \geq \beta \sqrt{d}$, then with probability at least $c_1(\beta)$, $\mathcal{SW}_2(\mu_m,\mu) \geq c_2(\beta) \sqrt{d/m}$. 
It is well-known that an isotropic, log-concave random vector satisfies that $\E\|X\|_2 \geq c\sqrt{d}$ for an absolute constant $c$, and thus \eqref{eq:lower-log-concave} is sharp up to the logarithmic factor.
\end{Remark}


Below we sketch the main ideas used in the proof of Theorem \ref{cor:Wasserstein.small.log-concave.subgaussian}. The complete proof can be found in Section \ref{sec:proofs.for.wasserstein}.

\subsection{Theorem \ref{cor:Wasserstein.small.log-concave.subgaussian}---highlights of the argument}
\label{sec:wasserstein.proof.log.concave.highlight}

As was noted previously,
\[
\mathcal{SW}_2(\mu_m,\mu)
=\sup_{\theta\in S^{d-1}} \left( \int_0^1 \left( F_{\mu^\theta_m}^{-1}(u) - F_{\mu^\theta}^{-1}(u) \right)^2 \,du \right)^{1/2}.
\]
Set $\delta=\kappa \Delta\log^2(\frac{e}{\Delta})$ for a well-chosen constant $\kappa\geq 1$ to be specified in what follows, let $U=[0,\delta)\cup(1-\delta,1]$, and observe that
\begin{align}
\label{eq:wasserstein.split}
\begin{split}
&\int_0^1 \left( F_{\mu^\theta_m}^{-1}(u) - F_{\mu^\theta}^{-1}(u) \right)^2 \,du \\
&\leq \int_{\delta}^{1-\delta} \left( F_{\mu^\theta_m}^{-1}(u) - F_{\mu^\theta}^{-1}(u) \right)^2 \,du
+ 2\int_U \left( F_{\mu^\theta_m}^{-1}(u) \right)^2 \,du + 2\int_U \left(  F_{\mu^\theta}^{-1}(u) \right)^2 \,du
\\
& = (1)+(2)+(3).
\end{split}
\end{align}

Intuitively, \eqref{eq:wasserstein.split} is sharp when $\Delta$ is small, as it is unrealistic to expect cancellations between  $F_{\mu^\theta_m}^{-1}(u)$ and $F_{\mu^\theta}^{-1}(u)$ when $u\in U$: that range corresponds to the ``$\delta$-outliers" of $\mu^\theta$ and the $\delta m$ extremal values of its empirical counterpart, respectively.

From here on the argument relies heavily on $\mu$'s log-concavity---and hence on the log-concavity of each $\mu^\theta$.

Consider first the effect of the outliers, i.e.,  $(2)$ and $(3)$ in Equation \eqref{eq:wasserstein.split}.
Clearly,
\begin{equation} \label{eq:log-concave-(3)}
(3) \leq c_0(\kappa) \Delta \log^4\left(\frac{1}{\Delta}\right)
\end{equation}
because marginals of a log-concave measure exhibit a sub-exponential tail-decay: setting 
\[\gamma(u)=\min\{u,1-u\} \quad\text{for }u\in[0,1],\]
it follows from Borell's Lemma (see, e.g., \cite{artstein2015asymptotic}) that for  $u \in (0,1)$,
\begin{align}
\label{eq:tail.log.concave}
\left| F_{\mu^\theta}^{-1}(u) \right|
\leq c_1 \log\left( \frac{1}{ \gamma(u)} \right).
\end{align}

Next, it is standard to show that $(2)$ is equivalent to $\frac{1}{m}\sum_{i=1}^{\delta m} (\inr{X_i,\theta}^\ast)^2$, assuming without loss of generality that $\delta m$ is an integer. Moreover, the behaviour of
\begin{align}
\label{eq:def.Hsm.intro}
H_{s,m}
=\sup_{\theta \in S^{d-1}} \max_{|I|=s} \left( \frac{1}{m}\sum_{i\in I} \inr{X_i,\theta}^2 \right)^{1/2}
=\sup_{\theta \in S^{d-1}} \left( \frac{1}{m}\sum_{i=1}^{s} (\inr{X_i,\theta}^\ast)^2 \right)^{1/2}
\end{align}
has been studied extensively over the years in rather general situations (see, e.g., \cite{ adamczak2010quantitative,mendelson2014singular,tikhomirov2018sample}).
In particular, when $X$ is isotropic and log-concave, the following estimate on \eqref{eq:def.Hsm.intro} was established in  \cite{adamczak2010quantitative} and also in \cite{talagrand2022upper}.

\begin{Theorem} \label{thm:tails.log.concave}
	There are  constants $c_1$ and $c_2$ depending only on $\kappa$  such that the following holds.
	Let  $m\geq 2d$ and set $\Delta\geq d/m$.
	Then with probability at least $1-2\exp(-c_2\sqrt{\Delta m}\log^2(\frac{1}{\Delta}))$,
	\[ H_{\delta m,m}
	\leq c_1 \sqrt\Delta \log^2\left(\frac{1}{\Delta}\right).
\]
\end{Theorem}

While the estimates on $(2)$ and $(3)$ are either well-understood or trivial, the key ingredient in the proof of Theorem \ref{cor:Wasserstein.small.log-concave.subgaussian} is the estimate on $(1)$: \emph{cancellations} that occur in the interval $[0,1] \backslash U$.
Those cancellations are due to a significant generalization of the Dvoretzky-Kiefer-Wolfowitz (DKW) inequality that is \emph{scale sensitive} and \emph{holds uniformly in $\theta$}. More accurately, we show that
with high probability,  uniformly in $\theta\in S^{d-1}$ and for any $t\in \R$ satisfying that $F_{\mu^\theta}(t)\in [\Delta,1-\Delta]$,
\begin{align}
\label{eq:distribution.empirica.true}
\left|F_{\mu^\theta_m}(t)-F_{\mu^\theta}(t)\right|
\leq \sqrt{\Delta  \gamma( F_{\mu^\theta}(t)) }\cdot \log\left(\frac{e}{\gamma(F_{\mu^\theta}(t) )} \right);
\end{align}
in particular,  for $u\in[\delta,1-\delta]$,
\begin{align} \label{eq:quantial.empirica.true}
  F_{\mu^\theta_m}^{-1}(u)
 \in \left[ F_{\mu^\theta}^{-1}\left( u-2\sqrt{\Delta \gamma(u)} \log\left(\tfrac{e}{\gamma(u)}\right) \right)
, F_{\mu^\theta}^{-1}\left( u+2\sqrt{\Delta \gamma(u)} \log\left(\tfrac{e}{\gamma(u)}\right)\right) \right].
\end{align}
The proofs of \eqref{eq:distribution.empirica.true} and \eqref{eq:quantial.empirica.true} can be found in Section \ref{subsection:ratio}.

Once \eqref{eq:quantial.empirica.true}  is established, one may invoke the fact that a log-concave measure with variance 1 satisfies Cheeger's isoperimetric inequality with an absolute constant \cite{bobkov1996extremal,kannan1995isoperimetric}: 
there is an absolute constant $h>0$ such that for any $u \in (0,1)$ and any $\theta \in S^{d-1}$,
\begin{align}
\label{eq:cheeger}
\frac{d}{du} F_{\mu^\theta}^{-1}(u)
\leq \frac{1}{h  \gamma(u) };
\end{align}
the formulation used here can be found in \cite{bobkov1999isoperimetric}.

By a first order Taylor expansion and in the high probability event on which \eqref{eq:quantial.empirica.true} holds,
\begin{align*}
&\int_{\delta}^{1-\delta} \left( F_{\mu^\theta_m}^{-1}(u) - F_{\mu^\theta}^{-1}(u) \right)^2  \,du \\
&\leq c_2 \int_{\delta}^{1-\delta} \left( \frac{  \sqrt{\Delta \gamma(u) } \log(\frac{e}{\gamma(u)})  }{ h  \gamma(u) } \right)^2 \,du
\sim  \Delta  \log^3\left(\frac{1}{\Delta}\right).
\end{align*}
Combining all these estimates it follows that
\[
\mathcal{SW}_2(\mu_m,\mu)
\leq C \sqrt \Delta \log^{2}\left(\frac{1}{\Delta}\right),
\]
as claimed.
\qed

\vspace{0.5em}
Let us stress that this argument relies on log-concavity in a crucial way and that the answer in the general case has to follow a completely different path.
Most notably, in the log-concave case $F_{\mu^\theta}^{-1}$ is differentiable with a well-behaved derivative, but for a general measure $F_{\mu^\theta}^{-1}$ need not even be continuous.
In particular, the pointwise control on $|F_{\mu^\theta_m}^{-1} - F_{\mu^\theta}^{-1}|$ one may use in the  log-concave case is simply false when it comes to an arbitrary measure. Cancellations in the integral on $[\delta,1-\delta]$ happen to be of a ``global" nature, leading to a weaker bound, that nevertheless is optimal in the context of Theorem \ref{thm:Wasserstein.small}.

\section{Theorem \ref{thm:Wasserstein.small} and Theorem \ref{cor:Wasserstein.small.log-concave.subgaussian} --- Proofs}
\label{sec:proofs.for.wasserstein}

We start with a word about \emph{notation}.
$\|\cdot\|_2$ is the Euclidean norm and  $\inr{\cdot,\cdot}$  is the standard inner product---though in what follows we do not specify the (finite) dimension of the underlying space.
Throughout, $c,c_0,c_1,C,C_0,C_1,\dots$ are absolute constants whose values may change from line to line.
If a constant $c$ depends on a parameter $a$, we write $c=c(a)$, and if $cA \leq B \leq CA$ for absolute constants $c$ and $C$, that is denoted by $A \sim B$.

\subsection{Preliminary estimates} 
\label{sec:preliminary}

In what follows, $X$ is a centred, isotropic random vector in $\R^d$ that satisfies $L_q-L_2$ \emph{norm equivalence} with constant $L$ for some $q\geq 4$. Hence,
\begin{align} \label{eq:def.norm.equiv}
 \|\inr{X,\theta}\|_{L_q} \leq L \|\inr{X,\theta}\|_{L_2} =L \quad\text{ for every } \theta\in S^{d-1}.
\end{align}
Let $\mu$ be the probability measure endowed by $X$, and for $\theta \in S^{d-1}$, $\mu^\theta$ is the marginal endowed by $\inr{X,\theta}$. 
Thus, for $t>0$
\[ \mu^\theta((-t,t)^c)
=\PP(|\inr{X,\theta}|\geq t)
\leq \frac{ \E|\inr{X,\theta}|^q}{t^q}
\leq \frac{L^q}{t^q}
\]
and it follows that for every $u\in(0,1)$
\begin{equation} \label{eq:Wasserstein.tail.decay.norm.equivalence}
\left| F_{\mu^\theta}^{-1}(u) \right| \leq \frac{L}{ \gamma(u)^{1/q} }.
\end{equation}
where we recall that $\gamma(u)=\min\{u,1-u\}$.

Let us turn to several simple observations that play an instrumental role in the proof of Theorem \ref{thm:Wasserstein.small}. 
We begin with the useful characterization of the Wasserstein distance between measures on the real line mentioned previously.

\begin{Lemma} \label{lem:Wasserstein.via.inverse.functions}
	For $\nu,\tau\in\mathcal{P}_2(\R)$,
	\begin{equation}
	\label{eq:wasserstein.inverse.function}
	\mathcal{W}_2(\nu,\tau)
	=  \left( \int_0^1 \left( F_\nu^{-1}(u) - F_\tau^{-1}(u) \right)^2\, du\right)^{1/2}.
	\end{equation}
\end{Lemma}

The proof of Lemma \ref{lem:Wasserstein.via.inverse.functions} can be found, for example, in \cite[Theorem 2]{ruschendorf1985wasserstein}. 
Let us sketch the simple argument for the upper estimate. 
If $|\cdot|$ is the uniform probability measure on $(0,1)$ and
	\[
\Pi(A)=| \{ u\in (0,1) : (F_\nu^{-1}(u), F_\tau^{-1}(u))\in A\}|,
\]
then $\Pi$ is  a coupling between $\nu$ and $\tau$. 
The $\mathcal{W}_2$ distance is the infimum over all such couplings and therefore,
	\[ \mathcal{W}_2^2(\nu,\tau)
	\leq \int_{\R\times\R} ( x-y )^2\, \Pi(dx,dy)
	= \int_0^1 \left( F_\nu^{-1}(u)- F_\tau^{-1}(u) \right)^2\, du .   \eqno\qed \]

The next observation is equally straightforward: the (trivial) connection between  $\mathcal{SW}_2$ and $\rho_{d,m}=\sup_{\theta\in S^{d-1}} | \|\Gamma\theta\|_2^2-1|$.

\begin{Lemma} \label{lem:Wasserstein.implies.Bai.Yin}
	If $\mathcal{SW}_2\left( \mu_m, \mu\right) \leq 1$ then $\rho_{d,m}\leq 3 \mathcal{SW}_2\left( \mu_m, \mu \right)$.
\end{Lemma}
\begin{proof}
	By the isotropicity of $X$, for $\theta \in S^{d-1}$ it is evident that
\begin{align*}
	\|\Gamma\theta\|_2^2-1
	&=\int_0^1  F^{-1}_{\mu^\theta_m}(u)^2\,du - \int_0^1  F^{-1}_{\mu^\theta}(u)^2\,du \\
	&= \int_0^1 \left( F^{-1}_{\mu^\theta_m}(u) -  F^{-1}_{\mu^\theta}(u) \right) \left( F^{-1}_{\mu^\theta_m}(u) +  F^{-1}_{\mu^\theta}(u) \right) \,du.
	\end{align*}
	The claim follows from the Cauchy-Schwartz inequality, \eqref{eq:wasserstein.inverse.function}, the triangle inequality, and the isotropicity of $X$.
\end{proof}

The final two observations focus on  connections between the Wasserstein distance and monotone ordering. 
Let $a^\sharp$ be the monotone non-decreasing rearrangement of a vector $a\in\R^m$, and for $\theta\in S^{d-1}$ and $1\leq i\leq m$, set $\lambda^\theta_i=m\int_{(i-1)/m}^{i/m} F_{\mu^\theta}^{-1}(u) \, du$.

\begin{Lemma} \label{lem:coordinate.distribution.concentration}
We have that
	\[
\sup_{\theta\in S^{d-1}} \left( \frac{1}{m}\sum_{i=1}^m \left( \inr{X_i,\theta}^\sharp - \lambda^\theta_i \right)^2  \right)^{1/2}
	\leq \mathcal{SW}_2(\mu_m,\mu).
\]
\end{Lemma}
\begin{proof}
	Note that for every $1\leq i\leq m$ and $ u\in ( \frac{i-1}{m},  \frac{i}{m}]$, $F_{\mu^\theta_m}^{-1}(u)=\inr{X_i,\theta}^\sharp$. Thus,  by Jensen's inequality and invoking Lemma \ref{lem:Wasserstein.via.inverse.functions},
	\begin{align*}
	\frac{1}{m}\sum_{i=1}^m \left( \inr{X_i,\theta}^\sharp - \lambda^\theta_i \right)^2
	&\leq \sum_{i=1}^m \int_{(i-1)/m}^{i/m} \left(  F_{\mu^\theta_m}^{-1}(u) - F_{\mu^\theta}^{-1}(u)\right)^2 \, du
	=\mathcal{W}_2^2\left(\mu^\theta_m,\mu^\theta\right). \qedhere
	\end{align*}
\end{proof}

From here on, $\kappa\geq 1$ is a (large) absolute constant that will be specified in what follows: assume that  $\Delta\leq (10 \kappa)^{-2}$, set
\begin{align*}
\delta&=\kappa\Delta \log^2\left(\frac{e}{\Delta}\right) \text{ and}\\
U&=[0,\delta)\cup(1-\delta,1] ;
\end{align*}
and observe that $\delta\leq\frac{1}{4}$.

Invoking \eqref{eq:wasserstein.inverse.function} and the triangle inequality,
\begin{align*}
\begin{split}
\mathcal{W}_2 \left( \mu^\theta_m,\mu^\theta \right)
	&\leq \left( \int_{\delta}^{1-\delta} \left( F_{\mu^\theta_m}^{-1}(u)- F_{\mu^\theta}^{-1}(u) \right)^2\,du \right)^{1/2} + \left(  \int_U F_{\mu^\theta_m}^{-1}(u)^2\,du \right)^{1/2}\\
	&\qquad   + \left(  \int_U F_{\mu^\theta}^{-1}(u)^2\,du \right)^{1/2} 
	=(1) + (2) + (3).
\end{split}
\end{align*}

The  term $(3)$  can be bounded trivially using the tail-estimate \eqref{eq:Wasserstein.tail.decay.norm.equivalence}.
In contrast, estimating $(2)$ is  more subtle, and we defer that to Section \ref{sec:sample.tails}.
The key estimate is on $(1)$.
It is based on our multi-dimensional,  scale sensitive generalization of the Dvoretzky-Kiefer-Wolfowitz inequality:
we show in Section \ref{subsection:ratio} that on a high probability event,
\[
F_{\mu^\theta_m}^{-1}(u)- F_{\mu^\theta}^{-1}(u)
\]
can be bounded (pointwise) using the (deterministic) difference
\begin{align}
\label{eq:det.diff}
F_{\mu^\theta}^{-1}(\psi_-(u))-F_{\mu^\theta}^{-1}(\psi_+(u)),
\end{align}
where $\psi_\pm$ are suitable perturbations of the identity.
In particular, the heart of the matter is to  control the  ``stability'' of $F^{-1}_{\mu^\theta}$---showing that \eqref{eq:det.diff} has a small $L_2$ norm.
It is important to note that this small $L_2$ norm \emph{does not} follow from an $L_\infty$ estimate on \eqref{eq:det.diff}; the latter is true only under strong regularity assumptions on $\mu$.

\subsection{A generalization of the DKW inequality for linear functions}
\label{subsection:ratio}
A key component in the proof of Theorem \ref{thm:Wasserstein.small} is the following uniform estimate on the deviation between the empirical and the true distribution functions for each marginal $\mu^\theta$.
Let
\[
\mathcal{U}
=\left\{ \{ \inr{\theta,\cdot}  \in I\}  : \theta \in \R^d, \ I \subset \R \ \text{is a generalized interval} \right\},
\]
where by ``generalized interval" we mean an open/closed, half-open/closed interval in $\mathbb{R}$, including rays.

\begin{Theorem}
\label{thm:ratio}
	There are absolute constants $c_0,c_1$ such that the following holds.
	Let  $\Delta\leq 1$ and $m\geq c_0 \frac{d}{\Delta}$.
	Then, with probability at least $1-\exp(-c_1 \Delta m)$, for every $A\in\mathcal{U}$ satisfying $\mu(A) \geq \Delta$, we have that
	\[
	\left| \mu_m (A) - \mu(A)   \right|
	\leq \sqrt{\Delta \mu(A) }  \log\left( \frac{e}{ \mu(A)} \right).
	\]
\end{Theorem}

\begin{Remark}
	A version of Theorem \ref{thm:ratio} was proved in \cite{mendelson2021approximating}, with an error estimate of the order  $\sqrt{\Delta\mu(A)}$ (without a logarithmic factor) but with the restriction that  $\Delta \geq c \frac{d}{m}\log(\frac{em}{d})$.
	That restriction results in a suboptimal estimate in Theorem \ref{thm:Wasserstein.small}, while the optimal one, at least for $q>4$, can be derived from  Theorem \ref{thm:ratio}.
	We conjecture that the logarithmic factor in Theorem \ref{thm:ratio} can be removed, but its removal does not affect the estimate in Theorem \ref{thm:Wasserstein.small} for $q>4$, nor will it eliminate the logarithmic factor when $q=4$.
\end{Remark}

The proof of Theorem \ref{thm:ratio} is a simple outcome of Talagrand's concentration inequality for classes of uniformly bounded functions \cite{talagrand1994sharper}, applied to a class of binary valued functions that has a finite $\mathrm{VC}$ dimension.
For more information on Talagrand's inequality see  \cite{boucheron2013concentration}, and detailed surveys on $\mathrm{VC}$-classes can be found in \cite{vaart1996weak}.

\begin{Definition}
	Let $\mathcal{H}$ be a class of subsets of $\Omega$.
	A set $\{x_1,\dots, x_\ell\}$ is \emph{shattered} by $\mathcal{H}$ if for every $I\subset\{1,\dots,\ell\}$ there exists $A\in\mathcal{H}$ such that $x_i\in A$ for $i\in I$ and $x_i\notin A$ for $i\notin I$.
	The \emph{${\rm VC}$-dimension} of $\mathcal{H}$ is 
	\[  {\rm VC}(\mathcal{H})
	=\sup\left\{ \ell \in\mathbb{N} : \{x_1,\dots,x_\ell\} \subset \Omega \text{ is shattered by } \mathcal{H} \right\}. \]
\end{Definition}

The following is an immediate corollary of Talagrand's concentration inequality combined with entropy estimates for VC-classes, see, e.g., \cite{boucheron2013concentration,ledoux2001concentration,vaart1996weak}.

\begin{Theorem}
\label{thm:talagrand.VC}
	There is an absolute constant $c$ such that the following holds.
	Let  $\mathcal{H}$ be a class of subsets of $\Omega$, put $k=\mathrm{VC}(\mathcal{H})$ and set
$\sigma^2=\sigma^2(\mathcal{H})=\sup_{A\in\mathcal{H}} \mu(A)$.
	Then, for every $t\geq 0$, with probability at least $1-\exp(-t)$,
\[ \sup_{A\in\mathcal{H}} |\mu_m(A)-\mu(A)|
\leq c\left( \sigma \sqrt{\frac{k}{m} \log\left(\frac{e}{\sigma } \right)  } + \frac{k}{m} \log\left(\frac{e}{\sigma } \right)   + \sigma \sqrt{\frac{t}{m}} + \frac{t}{m} \right) .\]
\end{Theorem}

\begin{proof}[Proof of Theorem \ref{thm:ratio}]
Using the notation introduced previously, set $j\geq 0$ such that $2^j\Delta\leq 1$, and  let
	\[ \mathcal{U}_j= \left\{ A \in\mathcal{U} :  \mu(A) \in [2^j\Delta, 2^{j+1} \Delta ) \right\}. \]
	It is well-known that as a class of `slabs' in $\R^d$, ${\rm VC}(\mathcal{U})\leq c_0 d$ for an absolute constant $c_0$ (see, e.g., \cite{vaart1996weak}) and in particular $\mathrm{VC}(\mathcal{U}_j) \leq c_0d$.
	Moreover, $\sigma^2(\mathcal{U}_j)\in[2^{j}\Delta,2^{j+1}\Delta]$, and 
by Theorem \ref{thm:talagrand.VC}, for every $t\geq 0$, with probability at least $1-2\exp(-t)$,
	\begin{align*}
	&\sup_{A\in\mathcal{U}_j} |\mu_m(A)-\mu(A)| \\
	&\leq c_1\left(  \sqrt{ 2^{j+1} \Delta \frac{d}{m} \log\left(\frac{e}{2^j \Delta } \right)  } + \frac{d}{m} \log\left(\frac{e}{ 2^j\Delta } \right)   + \sqrt{ 2^{j+1} \Delta} \sqrt{\frac{t}{m}} + \frac{t}{m} \right) .
	\end{align*}

	Set $t= c_2 \Delta m \log(\frac{e}{2^j\Delta})$ and let $m\geq c_3\frac{d}{\Delta}$.
	It is straightforward to verify that  with probability at least $1-2\exp(-c_2 \Delta m \log(\frac{e}{2^j\Delta}))$,
	\begin{align}
	\label{eq:VC.j}
	\sup_{A\in\mathcal{U}_j} |\mu_m(A) - \mu(A)|
	\leq \sqrt{\Delta 2^{j}\Delta} \log\left(\frac{e}{2^{j+1}\Delta}\right),
	\end{align}
	which is the required estimate for sets in $\mathcal{U}_j$.
	In particular, if $j_\Delta$ is the first integer such that $2^{j_\Delta} \Delta > 1$, then by the union bound and by comparing to a suitable geometric progression, with probability at least
	\begin{align*}
 1-\sum_{j=0}^{j_\Delta -1 } 2\exp\left( - c_2\Delta m \log\left(\frac{e}{2^j\Delta}\right) \right)
	\geq  1-2\exp(-c_4\Delta m),
	\end{align*}
	\eqref{eq:VC.j} holds for each $0\leq j\leq j_\Delta-1$.
\end{proof}
	
An immediate outcome of Theorem \ref{thm:ratio}  (applied to sets of the form $\{ \inr{\theta,\cdot } \leq t\}$ and $\{ \inr{\theta,\cdot } > t\}$) is the  following.

\begin{Corollary}
\label{cor:ratio.distribution.function}
	There are absolute constants $c_0,c_1$ such that the following holds.
	Let  $\Delta\leq 1$ and $m\geq c_0 \frac{d}{\Delta}$.
	Then, with probability at least $1-\exp(-c_1 \Delta m)$, for every $\theta\in S^{d-1}$ and $t\in\mathbb{R}$ satisfying that $F_{\mu^\theta}(t) \in[\Delta,1-\Delta]$, we have that
\begin{align}
\label{eq:ratio.estiamte}
	\left|F_{\mu^\theta_m}(t) - F_{\mu^\theta}(t)\right|
	\leq \sqrt{\Delta \gamma(F_{\mu^\theta}(t))  } \cdot \log\left( \frac{e}{ \gamma(F_{\mu^\theta}(t)) } \right).
\end{align}
\end{Corollary}

For the remainder of the proof of Theorem \ref{thm:Wasserstein.small}, we shall assume that $F_{\mu^\theta}$ is invertible.
This assumption is only made to simplify notation and holds without loss of generality.
Indeed, one may always consider $\nu$ distributed as  $\sqrt{1-\beta^2} X + \beta G$ where $G$ is the standard gaussian vector in $\R^d$ (independent of $X$) and $\beta$ is arbitrarily small.
In particular, $\nu$ is centred, isotropic, satisfies $L_q-L_2$ norm equivalence with a constant $L+1$, and $F_{\nu^\theta}$ is invertible for every $\theta\in S^{d-1}$. 
Thus, one may replace the RHS in \eqref{eq:ratio.estiamte}  by $\sqrt{2\Delta \gamma(F_{\nu^\theta}(t))}\log(\frac{e}{ \gamma( F_{\nu^\theta}(t)) })$, and $F_{\mu^\theta}$ by $F_{\nu^\theta}$ in all of the following arguments, finally taking $\beta$ to 0.

\vspace{0.5em}
By Lemma \ref{lem:Wasserstein.via.inverse.functions}, the key is to control the difference  between empirical and true \emph{inverse} distribution functions.
To that end, consider the two functions $\psi_+, \psi_- \colon[0,1]\to\R$ defined by
\[ \psi_\pm(u) = u  \pm 2\sqrt{ \Delta \gamma(u)} \log\left( \frac{e}{ \gamma(u) }  \right). \]
$\psi_\pm$ are perturbations of the identity, in a sense described in the next lemma.

\begin{Lemma}
\label{lem:psi.pm}
	There is an absolute constant $\kappa$ such that the following hold. Let $\delta=\kappa\Delta\log^2(\frac{e}{\Delta})$. 
	Then for $u\in [\delta,1-\delta]$ we have
	\begin{description}
	\item{$(1)$}  $\psi_{\pm}(u)\in [ \Delta,  1-\Delta ]$,
	\item{$(2)$} $|\psi_\pm(u)-u|\leq \frac{1}{10} \cdot \gamma(u)$,
	\item{$(3)$}  $\psi_\pm$ are absolutely continuous,  strictly increasing on $[\delta,1-\delta]$,  and their derivatives satisfy
	\[ \left| \psi_\pm'(u) - 1 \right|
	\leq 3 \sqrt\frac{\Delta}{ \gamma(u) } \log\left(\frac{e}{ \gamma(u) }\right).  \]
	\end{description}
\end{Lemma}
\begin{proof}
	We only consider the case $u\in[\delta, \frac{1}{2}]$ and thus $\gamma(u)=u$; the argument in the case $u\in[\frac{1}{2},1-\delta]$ is identical and is omitted.

	Note that if $\kappa\geq 1$ then $u\geq \Delta$ and in particular $\log\frac{e}{u}\leq \log\frac{e}{\Delta}$.
	Moreover, for $u\geq \delta=\kappa\Delta\log^2(\frac{e}{\Delta})$ we have that $\Delta \leq u/ \kappa \log^2(\frac{e}{\Delta})$. Hence,
	\begin{align*}
	 |\psi_\pm(u) - u|
	&= 2 \sqrt{\Delta u } \log\left( \frac{e}{u}\right)\\
	&\leq 2 \frac{u}{\sqrt{\kappa} \log(\frac{e}{\Delta}) }  \log\left( \frac{e}{\Delta} \right)
	\leq \frac{2 u}{\sqrt{\kappa}} ,
	\end{align*}
	and	(2) follows if $\kappa \geq 400$, while (1) is an immediate consequence of  (2) and the fact that $\delta\geq 2\Delta$ if $\kappa\geq 2$.
	
	Turning to (3), it is evident that $\psi_\pm$ are absolutely continuous and satisfy the claimed bound on their derivative.
	The fact that $\psi_\pm$ are strictly increasing on $[\delta,\frac{1}{2}]$ follows because $\Delta \leq  u / \kappa \log^2(\frac{e}{\Delta})$, and if $\kappa\geq  36$ then
	\[ \left| \psi_\pm'(u) - 1 \right|
	\leq 3 \sqrt\frac{\Delta}{ u} \log\left(\frac{e}{ u }\right)
	\leq \frac{3}{\sqrt{\kappa}}
	\leq \frac{1}{2}. 
	\qedhere\]
\end{proof}

From now on,  fix $\kappa$ as in Lemma \ref{lem:psi.pm}. By combining Corollary \ref{cor:ratio.distribution.function} and Lemma \ref{lem:psi.pm}, the following holds.

\begin{Lemma}
\label{lem:inverse.estimate}
	Fix a realization $(X_i)_{i=1}^m$ for which \eqref{eq:ratio.estiamte} holds.
	Then, for every $\theta\in S^{d-1}$ and every $u\in [\delta ,1-\delta ]$,
	\[ F_{\mu^\theta_m}^{-1}(u)
	\in \left[ F_{\mu^\theta}^{-1}\left( \psi_-(u) \right) ,  F_{\mu^\theta}^{-1}\left( \psi_+(u) \right) \right].
	\]
\end{Lemma}
\begin{proof}
	We only consider the case $u\in[\delta, \frac{1}{2}]$, and in particular $\gamma(u)=u$.
	The analysis when $u\in(\frac{1}{2},1-\delta]$ is identical and is omitted.
	
	First, let us show that $t_u=F_{\mu^\theta}^{-1} ( \psi_+(u) )$ satisfies that
	\begin{align}
	\label{eq:inverse.estimate.1}
	F_{\mu^\theta_m}(t_u)\geq u.
	\end{align}
	Note that if \eqref{eq:inverse.estimate.1} holds, then by the definition of the (right-)inverse, $F_{\mu^\theta_m}^{-1}(u)\leq F_{\mu^\theta}^{-1}(\psi_+(u))$.
	
	To prove \eqref{eq:inverse.estimate.1}, consider $u\in [\delta, \frac{1}{2}]$. Thus, $F_{\mu^\theta}(t_u)=\psi_+(u)$, and by Lemma \ref{lem:psi.pm}  $F_{\mu_\theta}(t_u)\in[\Delta,1-\Delta]$.
	In particular, by Corollary  \ref{cor:ratio.distribution.function}
\[F_{\mu^\theta_m}(t_u)
	\geq F_{\mu^\theta}(t_u) -  \sqrt{ \Delta  \gamma(F_{\mu^\theta}(t_u)) } \log\left(\frac{e}{ \gamma(F_{\mu^\theta}(t_u) ) } \right) .\]
	Applying Lemma \ref{lem:psi.pm} once again,    $\frac{3}{4}u\leq \gamma(F_{\mu^\theta}(t_u) )\leq \frac{5}{4}u$ and therefore
	\begin{align*}
	F_{\mu^\theta_m}(t_u)
	&\geq  u + 2\sqrt{\Delta u} \log\left(\frac{e}{u} \right) -  \sqrt{ \Delta  \tfrac{5}{4}  u } \log\left(\frac{4e }{ 3 u} \right) \\
	&= u + \sqrt{\Delta u } \left(2 \log\left(\frac{e}{u} \right)  -  \sqrt{  \tfrac{5}{4}  } \log\left(\frac{4e}{3 u} \right) \right).
	\end{align*}
	Now it is straightforward to verify that  $F_{\mu^\theta}(t_u)\geq u$, as claimed.
	
	Finally, using the same argument, if $t_u=F_{\mu^\theta}^{-1}( \psi_-(u) )$, then
	\[
F_{\mu^\theta_m}(t_u)<u
\]
	and therefore $F_{\mu^\theta_m}^{-1}(u)\geq F_{\mu^\theta}^{-1}(\psi_-(u))$.
\end{proof}

By Lemma \ref{lem:inverse.estimate} and the monotonicity of $F_{\mu^\theta}^{-1}$, we immediately have the following.

\begin{Corollary}
\label{cor:wasserstein.smaller.deterministic}
	Fix a realization $(X_i)_{i=1}^m$ that satisfies \eqref{eq:ratio.estiamte}.
	Then, for every $\theta\in S^{d-1}$,
\begin{align*}
 \int_{\delta}^{1-\delta} \left( F_{\mu^\theta_m}^{-1}(u)- F_{\mu^\theta}^{-1}(u) \right)^2 \,du
&\leq \int_{\delta}^{1-\delta} \left( F_{\mu^\theta}^{-1}\left( \psi_+(u) \right) - F_{\mu^\theta}^{-1}\left( \psi_-(u) \right) \right)^2 \,du.
\end{align*}
\end{Corollary}

\begin{Remark}
	It is worthwhile to note once again that the RHS in Corollary \ref{cor:wasserstein.smaller.deterministic}  depends only on  $F_{\mu^\theta}$ and not on $F_{\mu^\theta_m}$.
\end{Remark}

The analysis of
\[\int_{\delta}^{1-\delta} \left( F_{\mu^\theta}^{-1}\left( \psi_+(u) \right) - F_{\mu^\theta}^{-1}\left( \psi_-(u) \right) \right)^2 \,du\]
is the focus of the next section.

\subsection{A global modulus of continuity of $F_{\mu^\theta}^{-1}$}
\label{sec:wasserstein.deterministic}

Throughout this section we consider a distribution function $F$ of a symmetric random variable, and in particular  $F(0)=\frac{1}{2}$.
The formulations and proofs for general random variables require only trivial changes and are omitted for the sake of a simpler presentation.

Once again, we may and do assume without loss of generality  that $F$ is invertible, and that for some $q \geq 4$ and every $t \geq 1$, 
\begin{align}
\label{eq:deterministic.tail.decay}
 F(-t) \leq  \frac{L^q}{2 t^q}.
\end{align}

Before formulating the main estimate of this section, let us introduce some notation.
Set $\mathrm{id}\colon\R\to\R$ to be  the identity function, and denote the derivative of an absolutely continuous function $\psi\colon\R\to\R$ by $\psi'$.
For $q\geq 4$, let
\[
	r= 1+\frac{q}{2q-4},
\]
and recall that $\delta=\kappa\Delta\log^2(\frac{e}{\Delta})$ and $\gamma(u)=\min\{u,1-u\}$.
Since $\Delta \leq (10\kappa)^{-2}$ we have that $\delta \leq\frac{1}{4}$.

\begin{Proposition}
\label{prop:modulus.pertubation}
	Let
	\[\psi\colon[\delta,1-\delta]\to[0,1] \]
be an absolutely continuous, strictly increasing function that satisfies
$|\psi -\mathrm{id}|\leq \frac{1}{2} \gamma(\cdot)$, and for which either $\psi\geq \mathrm{id}$ or $\psi\leq \mathrm{id}$. Then there is a constant  $c=c(\kappa,q)$ such that
	\begin{align*}
	&\int_{\delta}^{1-\delta} \left( F^{-1}(u)-F^{-1}(\psi(u))\right)^2 \,du \\
	&\leq
	c
	\begin{cases}
	\sqrt{\Delta} +\left| \psi^{-1}\left( \tfrac{1}{2} \right) - \tfrac{1}{2} \right|  + \left(\int_{\delta}^{1-\delta} |\psi'(u)-1|^r\,du\right)^{1/r}
	&\text{if } q>4,
	\\
	\\
	\sqrt\Delta \log\frac{1}{\Delta} +\left| \psi^{-1}\left( \tfrac{1}{2} \right) - \tfrac{1}{2} \right|  +  \sqrt{\log\frac{1}{\Delta}} \left(\int_{\delta}^{1-\delta} |\psi'(u)-1|^2\,du\right)^{1/2}
	 &\text{if } q=4.
	\end{cases}
	\end{align*}
\end{Proposition}

Proposition \ref{prop:modulus.pertubation} replaces the key part of the argument that was presented in Section \ref{sec:wasserstein.proof.log.concave.highlight} and which was based on the assumption that $F^{-1}$ has a well behaved modulus of continuity.
That assumption allowed one to control $|F^{-1}(u)-F^{-1}(\psi(u))| $ in terms of $|u-\psi(u)|$, but unfortunately, it is useless in the general case. 
Instead, one may write
\begin{align*}
&\int_{\delta}^{1-\delta}\left( F^{-1}(u)-F^{-1}(\psi(u))\right)^2 \,du\\
&=\int_{\delta}^{1-\delta}  F^{-1}(u)^2 \,du + \int_{\delta}^{1-\delta} F^{-1}(\psi(u))^2 \, du -2\int_{\delta}^{1-\delta} F^{-1}(u)F^{-1}(\psi(u))\,du,
\end{align*}
and show that all three integrals are close to each other. The details of the proof are presented in Lemma \ref{lem:modulus.quadratic} and Lemma \ref{lem:modulus.cross}, but intuitively the reason that the three integrals are close is that $\psi$ is a perturbation of the identity. As a result, the second integral should be close to the first one. Moreover, if $u$ is sufficiently far away from the point of symmetry ($F(0)=\frac{1}{2}$), the terms $F^{-1}(u)$ and $F^{-1}(\psi(u))$ have the same sign; thus, by monotonicity, their product is sandwiched between $F^{-1}(u)^2$ and $F^{-1}(\psi(u))^2$.

\vspace{0.5em}
The formal proof is also based on the tail-estimate of $F$, namely that
\begin{align}
\label{eq:deterministic.tail.decay.inverse}
 \left| F^{-1}(u) \right|
 \leq \frac{L}{  \gamma(u)^{1/q}}
 \quad\text{for } u\in(0,1),
\end{align}
which is an immediate consequence of \eqref{eq:deterministic.tail.decay}.

\begin{Lemma}
\label{lem:modulus.quadratic}
	 There is a constant $c=c(\kappa,q,L)$  such that
	\begin{align*}
	&\left| \int_{\delta}^{1-\delta} F^{-1}(\psi(u))^2 \, du -  \int_{\delta}^{1-\delta} F^{-1}(u)^2\,du \right| \\
	&\leq
	c
	\begin{cases}
	\sqrt \Delta	 + \left(\int_{\delta}^{1-\delta} |\psi'(u)-1|^r \,du\right)^{1/r}  &\text{if } q>4,
	\\
	\\
	\sqrt\Delta \log \frac{1}{\Delta} + \sqrt{\log \frac{1}{\Delta}} \left(\int_{\delta}^{1-\delta} |\psi'(u)-1|^2 \,du\right)^{1/2}
	&\text{if } q=4,
	\end{cases}
	\end{align*}
where, as always, $r=1+\frac{q}{2q-4}$.
\end{Lemma}
\begin{proof}
	We only present the case $\psi\geq \mathrm{id}$ and $q>4$. The case $\psi\leq\mathrm{id}$ follows from an identical argument to the one presented here, while the case $q=4$ requires only simple modifications.
	
	By a change of variables $u\leftrightarrow \psi(u)$,
	\begin{align*}
	&\int_{\delta}^{1-\delta} F^{-1}(\psi(u))^2 \, du \\
	&=\int_{\delta}^{1-\delta} F^{-1}(\psi(u))^2\psi'(u)\, du - \int_{\delta}^{1-\delta} F^{-1}(\psi(u))^2(\psi'(u)-1) \, du \\
	&=\int_{\psi(\delta)}^{\psi(1-\delta)} F^{-1}(v)^2\, dv - \int_{\delta}^{1-\delta} F^{-1}(\psi(u))^2(\psi'(u)-1) \, du
	= A+B.
	\end{align*}
	
	First, observe that  $A$ is close to  $\int_{\delta}^{1-\delta} F^{-1}(u)^2\,du$.
	Indeed, since $ |\psi -\mathrm{id}|\leq \frac{1}{2}\gamma$ we have that
	\[ \psi(\delta)\in\left[\delta, 2\delta\right]
	\quad\text{and}\quad
	\psi(1-\delta)\in \left[1-\delta,1-\tfrac{1}{2}\delta \right].\]
	Hence, there is a constant  $c_1=c_1(\kappa,q,L)$ such that
	\begin{align*}
	\left| A - \int_{\delta}^{1-\delta} F^{-1}(v)^2\, dv  \right|
	&\leq \left( \int_{\delta }^{2\delta} +\int_{1-\delta}^{1-\delta/2} \right) F^{-1}(v)^2\,dv
	\leq c_1 \sqrt\Delta
	\end{align*}
	where the last inequality  follows from  the tail-estimate \eqref{eq:deterministic.tail.decay.inverse}, using that $q>4$.
	
	Second, to estimate $B$,  apply  H\"older's inequality (with exponent $r$) 
	\begin{align*}
	B&\leq \left( \int_{\delta}^{1-\delta} |F^{-1}(\psi(u))|^{2r'} \, du\right)^{1/r'} \left( \int_{\delta}^{1-\delta} |\psi'(u)-1|^{r} \, du\right)^{1/r}
	= B_1\cdot B_2.
	\end{align*}
	To complete the proof it suffices to show that  $B_1\leq c_1(q,L)$.
	To that end, note that  $|\psi-\rm id|\leq \frac{1}{2}\gamma$.
	Thus, by the  tail-estimate \eqref{eq:deterministic.tail.decay.inverse} there is an absolute constant $c_2$ such that for any  $u\in[\delta,1-\delta]$,
\[ 
\left|F^{-1}(\psi(u))\right| \leq \frac{ c_2 L }{ \gamma(u)^{1/q}  } .
\]
	Finally, with the choice of $r$  we have that $2r'/q<1$; hence $\gamma(\cdot)^{-2r'/q}$ is integrable in $(0,1)$ and $B_1\leq c_1(q,L)$.
\end{proof}

\begin{Lemma}
\label{lem:modulus.cross}
	There is a  constant $c=c(\kappa,q,L)$ such that
	\begin{align*}
	&\left| \int_{\delta }^{1-\delta } F^{-1}(u)F^{-1}(\psi(u))\,du -  \int_{\delta}^{1-\delta} F^{-1}(u)^2 \,du \right| \\
	&\leq
	c\begin{cases}
	\sqrt \Delta + \left| \psi^{-1}(\tfrac{1}{2}) - \tfrac{1}{2} \right| +  \left(\int_{\delta }^{1-\delta } |\psi'(u)-1|^r \,du\right)^{1/r}  &\text{if } q>4, \\
	\\
	\sqrt \Delta \log \frac{1}{\Delta}+ \left| \psi^{-1}(\tfrac{1}{2}) - \tfrac{1}{2} \right| +   \sqrt{\log\frac{1}{\Delta}} \left(\int_{\delta}^{1-\delta} |\psi'(u)-1|^2 \,du\right)^{1/2}
	&\text{if } q=4.
	\end{cases}
	\end{align*}
\end{Lemma}
\begin{proof}
Once again, we only present the proof in the case  $\psi\geq \rm id$ and $q>4$.
	The other cases follow a similar path and are omitted.

	If $\psi\geq \rm id$ then $\psi^{-1}(\frac{1}{2})\leq \frac{1}{2}$, and since $|\psi-\mathrm{id}|\leq \frac{1}{2}\rm{id}$ then $\psi^{-1}(\frac{1}{2})\geq \frac{1}{4}\geq \delta$.
	Now set
	\begin{align*}
	&\int_{\delta }^{1-\delta } F^{-1}(u)F^{-1}(\psi(u))\,du  \\
	&=\left( \int_{\delta }^{\psi^{-1}(1/2)} + \int_{\psi^{-1}(1/2)}^{1/2} + \int_{1/2}^{1-\delta } \right) F^{-1}(u)F^{-1}(\psi(u))\,du   \\
	&=A+B+D.
	\end{align*}
	
	For $u$ in the range of integration of $B$, we have that $u,\psi(u)\in [\frac{1}{4},\frac{3}{4}]$. 
	By the tail-estimate \eqref{eq:deterministic.tail.decay.inverse}, there is an absolute constant $c_1$ such that for $v\in[\frac{1}{4},\frac{3}{4}]$,  $|F^{-1}(v)|\leq c_1 L$; thus 	$B	\leq (c_1L)^2 |  \psi^{-1}(\tfrac{1}{2}) - \tfrac{1}{2} | $.
	
	Turning to  $A$, let us show that
	\begin{align}
	\label{eq:estimate.A.deterministic}
	\begin{split}
	\left| A- \int_{\delta }^{1/2}  F^{-1}(u)^2\,du \right|
	&\leq c_2\left( \sqrt \Delta  +  \left| \psi^{-1}\left(\tfrac{1}{2}\right) - \tfrac{1}{2}  \right|  + \left(\int_{\delta}^{1-\delta} |\psi'(u)-1|^r \,du\right)^{1/r} \right)
	\end{split}
	\end{align}
	for a constant $c_2=c_2(\kappa,q,L)$.
	
Note that $F(\frac{1}{2})=0$ since the underlying random variable is symmetric.  
	Thus, by the  monotonicity of $F^{-1}$ and recalling that $\psi\geq \mathrm{id}$,
	\[ F^{-1}(u)\leq F^{-1}(\psi(u))\leq 0
	\quad\text{for } u\in\left[\delta,\psi^{-1}(\tfrac{1}{2}) \right].\]
	Setting
	\begin{align*}
	A_-
	&=\int_{\delta }^{\psi^{-1}(1/2)} F^{-1}(\psi(u))^2\,du	\quad\text{and}\\
	A_+
	&= \int_{\delta }^{\psi^{-1}(1/2)} F^{-1}(u)^2\,du,
	\end{align*}
	it is evident that $A\in[A_-,A_+]$. Therefore, it suffices to show that $A_-$ and $A_+$ are both sufficiently close to $\int_{\delta }^{1/2} F^{-1}(u)^2\,du$.
	
	Using once again that $|F^{-1}(u)| \leq c_1L$ for $u\in[\psi^{-1}(\frac{1}{2}),\frac{1}{2}]$, it is evident that
	\[ \left| A_+ - \int_{\delta }^{1/2}  F^{-1}(u)^2\,du \right|
	\leq (c_1L)^2 \left| \psi^{-1}\left(\tfrac{1}{2}\right) - \tfrac{1}{2}  \right| .\]
	As for $A_-$, one may follow the same argument as used in the proof of Lemma \ref{lem:modulus.quadratic}---a change of variables, H\"older's inequality, and invoking \eqref{eq:deterministic.tail.decay.inverse}---to show that
	\begin{align*}
	&\left| A_- -   \int_{\delta }^{1/2}  F^{-1}(u)^2\,du  \right|
	\leq c_4 \left( \sqrt \Delta	 + \left(\int_{\delta }^{1-\delta} |\psi'(u)-1|^r \,du\right)^{1/r} \right)
	\end{align*}
	for a constant $c_4=c_4(\kappa,q,L)$.
	This proves \eqref{eq:estimate.A.deterministic}

	Finally, an  identical argument can be used to show that
	\begin{align*}
	\left| D- \int_{1/2}^{1-\delta }  F^{-1}(u)^2\,du \right|
	&\leq c_5\left(  \sqrt \Delta + \left| \psi^{-1}\left(\tfrac{1}{2}\right) - \tfrac{1}{2}  \right|  + \left(\int_{\delta}^{1-\delta } |\psi'(u)-1|^r \,du\right)^{1/r} \right)
	\end{align*}
	for a constant $c_5=c_5(\kappa,q,L)$, completing the proof.
\end{proof}

\subsection{The deterministic estimate}
\label{sec:det.error.explicit}
Recall that
\[ \psi_\pm(u) = u  \pm 2\sqrt{ \Delta \gamma(u) } \log\left( \frac{e}{ \gamma(u) }  \right). \]
With Corollary \ref{cor:wasserstein.smaller.deterministic} in mind, the following estimate is crucial:

\begin{Lemma}
\label{lem:det.integrals.pertubations}
	There exists an absolute constant $\kappa$ and a constant $c(\kappa,q,L)$ such that
	\[
	\int_{\delta}^{1-\delta} \left( F_{\mu^\theta}^{-1}\left( \psi_-(u)  \right) - F_{\mu^\theta}^{-1}\left( \psi_+(u) \right) \right)^2 \,du
	\leq c
	\begin{cases}
		\sqrt{\Delta} &\text{if } q>4,\\
		\\
	 \sqrt{\Delta} \log^2 \frac{1}{\Delta}
	 &\text{if } q=4.
	\end{cases}\]
\end{Lemma}

Again, we shall assume for the sake of simplicity that each $F_{\mu^\theta}$ is symmetric and without loss of generality that it is invertible. 	The argument is based on applying Proposition \ref{prop:modulus.pertubation} to  $F=F_{\mu^\theta}$ and the monotone functions $\psi_+$ and $\psi_-$.

\begin{proof}
	 Clearly, $F$ falls within the scope of Proposition \ref{prop:modulus.pertubation}.
	 Moreover $\psi_+\geq\rm id$ and $\psi_-\leq \rm id$, and if $\kappa$ is as in  Lemma \ref{lem:psi.pm}, then by that lemma $\psi_{\pm}$ satisfy the remaining conditions in Proposition \ref{prop:modulus.pertubation}. 
	 Hence, for the choice  $r=1 + \frac{q}{2q-4}$, there is a constant $c_1=c_1(\kappa,q,L)$ such that 
	\begin{align*}
	&\int_{\delta}^{1-\delta} \left( F_{\mu^\theta}^{-1}\left( \psi_\pm (u)  \right) - F_{\mu^\theta}^{-1}\left( u \right) \right)^2 \,du \\
	&\leq c_1
	\begin{cases}
	\sqrt{\Delta} +\left| \psi^{-1}_\pm\left( \tfrac{1}{2} \right) - \tfrac{1}{2} \right|  + \left(\int_{\delta}^{1-\delta} |\psi_\pm'(u)-1|^r\,du\right)^{1/r}
	&\text{if } q>4,
	\\
	\\
	\sqrt\Delta \log\frac{1}{\Delta} +\left| \psi_\pm^{-1}\left( \tfrac{1}{2} \right) - \tfrac{1}{2} \right|  +  \sqrt{\log\frac{1}{\Delta}} \left(\int_{\delta}^{1-\delta} |\psi_\pm'(u)-1|^2\,du\right)^{1/2}
	 &\text{if } q=4.
	\end{cases}	
	\end{align*}
	
	Clearly $|\psi_\pm ^{-1}(\frac{1}{2})-\frac{1}{2}|\leq 10 \sqrt \Delta$, and by Lemma \ref{lem:psi.pm}
	\[  \left|\psi'_\pm(u)-1 \right|
	\leq 3\sqrt{\frac{\Delta}{ \gamma(u) }} \log\left(\frac{e}{  \gamma(u) } \right).\]
	Thus
	\[  \left(\int_\delta^{1-\delta} \left| \psi'_\pm(u)-1 \right|^r \,du\right)^{1/r}
	\leq
	c_2
	\begin{cases}
	\sqrt\Delta &\text{if } q>4,\\
	\sqrt\Delta \log^{3/2}\frac{1}{\Delta} &\text{if } q=4,
	\end{cases}
	\]
	for a suitable constant $c_2=c_2(\kappa,q)$.
The proof is completed by an application of the $L_2$ triangle inequality.
\end{proof}

\subsection{The empirical tail integrals}
\label{sec:sample.tails}

The last component needed in the proofs of Theorem \ref{thm:Wasserstein.small} and \ref{cor:Wasserstein.small.log-concave.subgaussian} is an estimate on 
\begin{align}
\label{eq:sample.tails}
\sup_{\theta\in S^{d-1}} \left( \int_{U} \left( F_{\mu^\theta_m}^{-1}(u)\right)^2\,du \right)^{1/2}
\end{align}
where  $U=[0,\delta)\cup(1-\delta,1]$.

We begin with a general estimate in terms of
$\rho_{d,m}=\sup_{\theta\in S^{d-1}} | \|\Gamma\theta\|_2^2-1|$ that suffices for the proof of Theorem \ref{thm:Wasserstein.small}.

\begin{Lemma}
\label{lem:tail.integral.bounded.by.Bai.Yin.plus.ratio}
	There are absolute constants $\kappa,c_0,c_1$ and a constant $c_2=c_2(\kappa,q,L)$ such that the following holds.
	If $m\geq c_0 \frac{d}{\Delta}$, then with probability at least $1-\exp(-c_1\Delta m)$,
	\begin{align*}
	\sup_{\theta\in S^{d-1}} \left( \int_{U} \left( F_{\mu^\theta_m}^{-1}(u)\right)^2\,du \right)^{1/2}
	&\leq  c_2
	\begin{cases}
	\rho_{d,m}^{1/2} + \Delta^{1/4}
	&\text{if } q>4, \\
	\\
	\rho_{d,m}^{1/2}  + \Delta^{1/4} \log\frac{1}{\Delta}
	&\text{if } q=4.
	\end{cases}
	\end{align*}
\end{Lemma}
\begin{proof}
	Fix $\theta\in S^{d-1}$ and set
	\begin{align*}
	A_\theta&=\int_0^1  F_{\mu^\theta_m}^{-1}(u)^2\,du - \int_{0}^1 F_{\mu^\theta}^{-1}(u)^2\,du,\\
	B_\theta&= \int_{\delta}^{1-\delta}  F_{\mu^\theta}^{-1}(u)^2\,du - \int_{\delta}^{1-\delta}  F_{\mu^\theta_m}^{-1}(u)^2\,du.
	\end{align*}
	Thus
	\begin{align*}
	\int_{U}  F_{\mu^\theta_m}^{-1}(u)^2\,du
	&= \int_U  F_{\mu^\theta}^{-1}(u)^2\,du + A_\theta + B_\theta .
	\end{align*}
	By isotropicity, $A_\theta=\|\Gamma\theta \|_2^2 -1$, and in particular $\sup_{\theta\in S^{d-1}}|A_\theta|\leq  \rho_{d,m}$.
	Moreover, by \eqref{eq:deterministic.tail.decay.inverse},
	\[ \int_U  F_{\mu^\theta}^{-1}(u)^2\,du
	\leq 2 \int_0^\delta \frac{L^2}{u^{2/q}} \,du
	\leq c_1
	\begin{cases}
	\sqrt{\Delta} &\text{if } q>4,\\
	\sqrt\Delta \log\frac{1}{\Delta} &\text{if } q=4,
	\end{cases}\]
	for a constant $c_1=c_1(\kappa,q,L)$.

	As for $B_\theta$, using the estimates from Section \ref{sec:wasserstein.deterministic} and Section \ref{sec:det.error.explicit} one may verify that with probability at least $1-\exp(-c_2\Delta m)$,
	\begin{align}
	\sup_{\theta\in S^{d-1}}|B_\theta|
	\leq c_3
	\begin{cases}
	\sqrt \Delta
	&\text{if } q>4,\\
	\sqrt\Delta \log^2\frac{1}{\Delta}
	&\text{if } q=4 ,
	\end{cases}
	\end{align}
	for a constant $c_3=c_3(\kappa,q,L)$.	
\end{proof}

Lemma \ref{lem:tail.integral.bounded.by.Bai.Yin.plus.ratio} implies that if $\rho_{d,m}\leq \sqrt\Delta$ then \eqref{eq:sample.tails} is (at most) of order $\Delta^{1/4}$---up to a  logarithmic factor. 
While this estimate is sufficient for the proof of Theorem \ref{thm:Wasserstein.small}, it is not enough when the goal is to obtain an upper bound of $\Delta^{1/2}$--- when such an estimate is possible as in Theorem \ref{cor:Wasserstein.small.log-concave.subgaussian}.

That calls for a  more careful analysis of \eqref{eq:sample.tails}, and to that end, let us  re-write \eqref{eq:sample.tails} in  a more standard form.

\begin{Definition}
\label{def:S}
	For $1\leq s\leq m$, let
	\[ H_{s,m}=\sup_{\theta \in S^{d-1}} \max_{  |I|= s} \left(\frac{1}{m} \sum_{i \in I} \inr{X_i,\theta}^2 \right)^{1/2} .\]
	Thus,  $\sqrt{m}H_{s,m}$ is the Euclidean norm of the largest $s$ coordinates of $(|\inr{X_i,\theta}|)_{i=1}^m$  taken in the `worst' direction $\theta$.
\end{Definition}

It is straightforward to verify that \eqref{eq:sample.tails} is equivalent to $H_{\delta m ,m}$. Also, sharp estimates on $H_{s,m}$ have been established as part of the study of $\rho_{d,m}$, (see, e.g., \cite{adamczak2010quantitative,bartl2022random,mendelson2014singular,tikhomirov2018sample}).
Let us outline two such cases:

\begin{Example}
\label{ex:S.subgaussian}
	If $X$ is $L$-subgaussian, then there are constants $c_1,c_2$ depending only on $L$ such that, with probability at least $1-2\exp(-c_1 s\log (em/s))$,
	\begin{align*}
	H_{s,m}
	\leq c_2\left(  \sqrt \frac{d}{m} + \sqrt{ \frac{s}{m} \log \left( \frac{em}{s}\right)}  \right).
	\end{align*}
\end{Example}

	The proof of Example \ref{ex:S.subgaussian} is standard and follows from a net argument and individual  subgaussian tail-decay, see e.g., \cite{bartl2022nongaussian}.

\begin{Example}
\label{ex:S.logconcave}
	If $X$ is a $\psi_1$ random vector (that is, it satisfies $L_q-L_2$ norm equivalence with constant $qL$ for every $q\geq 2$), then there are  constants $c_1,c_2$  depending on $L$, such that with probability at least $1-2\exp(-c_1\sqrt s \log (em/s))$,
	\[	H_{s,m}
	\leq c_2\left(  \max_{1\leq i \leq m} \frac{  \|X_i\|_2 }{ \sqrt m}  +  \sqrt \frac{s}{m} \log \left( \frac{em}{s} \right) \right).\]
\end{Example}

Example \ref{ex:S.logconcave} was  established  \cite{adamczak2010quantitative}, with a minor (but from our perspective important) restriction: an upper bound on $m$ as a function of $d$.
This restriction was later removed in \cite{talagrand2022upper}, see Theorem 14.3.1 and Proposition 14.3.3 therein.

Theorem \ref{thm:tails.log.concave} follows by combining Example \ref{ex:S.logconcave} and a well-known result on the tail-decay of the Euclidean norm of an isotropic log-concave random vector, due to  Paouris \cite{paouris2006concentration}.

\begin{Theorem} \label{thm:Paouris}
There is an absolute constant $c$ such that if $X$ is an isotropic log-concave random vector in $\R^d$ and $u \geq 1$, then
\[
\PP\left( \|X\|_2\geq c u \sqrt d \right)
	\leq 2\exp\left( -u\sqrt d \right).
\]
\end{Theorem}

\begin{proof}[Proof of Theorem \ref{thm:tails.log.concave}]
By Borell's lemma, an isotropic log-concave random vector is $\psi_1$ with an absolute constant $L$. Thus, following Example \ref{ex:S.logconcave}, it remains show that there are absolute constants $c_1,c_2$ such that with probability at least $1-2\exp(-c_1\sqrt{\Delta m}\log^2(\frac{1}{\Delta}))$,
	\[ \max_{1\leq i\leq m} \frac{\|X_i\|_2}{\sqrt m} \leq c_2 \sqrt\Delta \log^2 \left(\frac{1}{\Delta}\right). \]
Let $u=\beta \sqrt\frac{\Delta m}{d} \log^2 (\frac{1}{\Delta})$ and by Theorem \ref{thm:Paouris} and the union bound,
	\[\PP\left(\max_{1\leq i\leq m} \frac{\|X_i\|_2}{\sqrt m} \geq c_3 \beta \sqrt\Delta \log^2 \left(\frac{1}{\Delta} \right) \right)
	\leq 2m\exp\left( -\beta \sqrt{\Delta m} \log^2\left(\frac{1}{\Delta}\right) \right). \]
	Hence, recalling that  $\Delta \geq d/m$ and $m\geq 2d$, the claim follows if $\beta$ is a sufficiently large absolute constant. 
\end{proof}

\subsection{Proofs of Theorems \ref{thm:Wasserstein.small} and \ref{cor:Wasserstein.small.log-concave.subgaussian}}

Let us connect all the steps we have made, leading to the proofs of Theorems \ref{thm:Wasserstein.small} and \ref{cor:Wasserstein.small.log-concave.subgaussian}. 

\vspace{0.5em}
\noindent Set $\kappa$ as in Lemma \ref{lem:psi.pm}. 
Note that for an absolute constant $c_1$ and a constant $c_2=c_2(L)$,
\[
A=\sup_{\theta\in S^{d-1}}\int_U \left( F_{\mu^\theta}^{-1}(u) \right)^2 \,du \leq
\begin{cases}
c_1 \Delta \log^4\frac{1}{\Delta} & \text{ if $X$ is log-concave,}\\
c_2 \Delta \log^3\frac{1}{\Delta} & \text{ if $X$ is $L$-subgaussian.}
\end{cases}
  \]
And in the general case, when $X$ only satisfies $L_q-L_2$ norm equivalence with a constant $L$,  there is a constant $c_3=c_3(q,L)$ such that $A\leq c_3 \sqrt{\Delta}$ if $q>4$ and $A\leq c_3 \sqrt\Delta \log (\frac{1}{\Delta})$ if $q=4$.

\vspace{0.5em}
\noindent
The empirical tail integral $\sup_{\theta\in S^{d-1}}\int_U ( F_{\mu^\theta_m}^{-1}(u) )^2 \,du$ is controlled by Theorem \ref{thm:tails.log.concave} (applied to $s=\delta m$) in the log-concave case;  by Example  \ref{ex:S.subgaussian} (again applied to $s=\delta m$) in the subgaussian case; and by  Lemma \ref{lem:tail.integral.bounded.by.Bai.Yin.plus.ratio} in the general case.

\vspace{0.5em}
\noindent
As for $\sup_{\theta\in S^{d-1}}\int_\delta^{1-\delta} ( F_{\mu^\theta_m}^{-1}(u)-F_{\mu^\theta}^{-1}(u) )^2 \,du$,  fix a realization $(X_i)_{i=1}^m$ for which \eqref{eq:ratio.estiamte} holds.
The log-concave case follows from the arguments presented in Section \ref{sec:wasserstein.proof.log.concave.highlight}, and the general case follows from Corollary \ref{cor:wasserstein.smaller.deterministic} and Lemma \ref{lem:det.integrals.pertubations}.
\qed

\vspace{1em}

Finally, let us record the following (trivial) bound and its highly useful consequence---obtained by invoking Theorem \ref{thm:Wasserstein.small}. 
Observe that if $X$ is isotropic and satisfies $L_4-L_2$ norm equivalence with a constant $L$, then by the triangle inequality and the tail-estimate \eqref{eq:deterministic.tail.decay.inverse} we have that for every $1\leq s\leq m$,
	\[
H_{s,m} \leq \mathcal{SW}_2(\mu_m,\mu) + c(L) \left(\frac{s}{m}\right)^{1/4}.
\]
Setting $\Delta=cd/m$ in Theorem \ref{thm:Wasserstein.small} leads to the following:

\begin{Corollary} 
\label{rem:bound.H.via.S.W.trivial}
Assume that 
\begin{description}
\item{$(1)$} $X$ is isotropic and satisfies $L_4-L_2$ norm equivalence with  constant $L$; and
\item{$(2)$} there is some $\alpha \geq 4$ such that with probability $1-\eta$, $
\rho_{d,m} \leq \tilde{c} (d/m)^{2/\alpha}$.
\end{description}
Then there are constants $c_1,\dots,c_4$ that depend on $L$ and $\tilde{c}$ such that, if $m\geq c_1 d$, then with probability at least $1-\eta-\exp(-c_2 d)$ both
\[
\mathcal{SW}_2(\mu_m,\mu)
\leq c_3 \left(\frac{d}{m}\right)^{1/\alpha} \log\left(\frac{m}{d}\right), \]
and for every $1\leq s\leq m$,
\[
H_{s,m} \leq c_4 \left( \left(\frac{s}{m}\right)^{1/4} +  \left(\frac{d}{m}\right)^{1/\alpha} \log\left(\frac{m}{d}\right) \right)  .
\]
\end{Corollary}

Note that the case $\alpha=4$ corresponds to the quantitative Bai-Yin estimate, namely that with high probability, $\rho_{d,m} \leq \tilde{c} \sqrt{d/m}$.

\subsection{Additional proofs}

For $\theta\in S^{d-1}$ and  $i=1,\dots, m-1$, denote by $q_i^\theta=F_\theta^{-1}(\frac{i}{m})$ the $\frac{i}{m}$-quantile of  $\mu^\theta$, and  set $q_m^\theta=q_{m-1}^\theta$.
Recall that $\lambda_i^\theta = \int_{(i-1)/m}^{i/m} F_\theta^{-1}(u)\, du$ and thus 
\[q_{i-1}^\theta\leq \lambda_i^\theta\leq q_{i}^\theta \quad\text{for }1< i <m.\]

\begin{Lemma}
\label{lem:Gamma.theta.quantils.averaged.vs.not}
	Let $X$ be centred and isotropic and let $m\geq 4$.
	If either 
	\begin{enumerate}
	\item[(A1)] $\sup_{\theta \in S^{d-1}} \|\inr{X,\theta}\|_{L_4} \leq L$ for some $L\geq 1$, or
	\item[(A2)] $X$ is log-concave,
	\end{enumerate}
	then there is a constant $c_1$ that depends only on $L$ and an absolute constant $c_2$ such that
	\begin{align*}
	&  \left| \sup_{\theta\in S^{d-1}} \left( \frac{1}{m}\sum_{i=1}^m \left( \inr{X_i,\theta}^\sharp - q_i^\theta\right)^2 \right)^{1/2} -  \sup_{\theta\in S^{d-1}} \left( \frac{1}{m}\sum_{i=1}^m \left( \inr{X_i,\theta}^\sharp - \lambda_i^\theta\right)^2 \right)^{1/2} \right| \\
	&\leq 
	\begin{cases}
	 c_1 / m^{1/4} & \text{if  (A1) holds},\\
	c_2 \log(m)/\sqrt{m} & \text{if  (A2) holds}.
 	 \end{cases}
 	 \end{align*}	 
\end{Lemma}
\begin{proof}
	By the triangle inequality, it suffices to prove that
	\[\sup_{\theta\in S^{d-1}}\|\lambda^\theta-q^\theta\|_2 \leq \begin{cases}
	 c_1 m^{1/4} & \text{if  (A1) holds},\\
	c_2 \log(m) & \text{if  (A2) holds}.
 	 \end{cases}\]
 	 To that end, fix $\theta\in S^{d-1}$.
	Since $q_{i-1}^\theta\leq \lambda_i^\theta\leq q_{i}^\theta$ for $1<i<m$,
	\[\|\lambda^\theta-q^\theta\|_2^2
	 \leq \sum_{i=2}^{m-1} (q^\theta_{i-1}- q^\theta_{i})^2 + 2\left( (\lambda^\theta_1)^2 + (q^\theta_1)^2 + (\lambda^\theta_m)^2 + (q^\theta_m)^2 \right). \]
	The four terms $\lambda^\theta_1,q^\theta_1,\lambda^\theta_m,q^\theta_m$ can be bounded using the tail-estimate on $F^{-1}_{\mu^\theta}$.
	Specifically, under Assumption (A1) we have  that $|\lambda_1^\theta|\leq c_3(L)m^{1/4}$ (see \eqref{eq:Wasserstein.tail.decay.norm.equivalence}), whereas under Assumption (A2) we have that $|\lambda_1^\theta|\leq c_4\log(m)$  (see \eqref{eq:tail.log.concave}). Similar bounds are true for the terms $ q^\theta_1,\lambda^\theta_m,q^\theta_m$.
	Thus, in both cases all that remains is to estimate 
	\[\sum_{i=2}^{m-1} (q^\theta_{i-1}- q^\theta_{i})^2.\]
	
	Case 1 --- Assumption (A1) holds.
	
In that case, 
	\[\sum_{i=2}^{m-1} (q^\theta_{i-1}- q^\theta_{i})^2
	=\sum_{i=2}^{m-1} (q^\theta_{i-1})^2- 2 \sum_{i=2}^{m-1} q^\theta_{i-1} q^\theta_{i} + \sum_{i=2}^{m-1} (q^\theta_{i})^2
	=A -2 B + D. \]
	By the tail-estimate \eqref{eq:Wasserstein.tail.decay.norm.equivalence}, $|A+D- 2\|q^\theta\|_2^2| \leq c_5(L)\sqrt{m}$.
	Thus, it suffices to show that $|B - \|q^\theta\|_2^2|\leq c_6(L)\sqrt m$.
	To that
end, set $i_{0}$ to be the largest integer smaller than or equal to 
$\frac{m}{2}$. Since $m\geq 4$ we have
$\frac{i_{0}}{m} \in [\frac{2}{5},\frac{1}{2}]$ and by
the tail-estimate \eqref{eq:Wasserstein.tail.decay.norm.equivalence},
$|q^{\theta}_{i_{0}} q^{\theta}_{i_{0}+1}|\leq c_{7}(L)$. Moreover,
$X$ is symmetric and therefore $q^{\theta}_{i}\leq 0$ if
$i\leq i_{0}$ and $q^{\theta}_{i}\geq 0$ if $i>i_{0}$. Using the monotonicity
of $i\mapsto q^{\theta}_{i}$,
\begin{align*}
B & \geq \sum _{i=2}^{i_{0}} q^{\theta}_{i-1} q^{\theta}_{i} +
\sum _{i=i_{0}+2}^{m-1} q^{\theta}_{i-1} q^{\theta}_{i} - c_{7}
\\
&\geq \sum _{i=2}^{i_{0}} (q^{\theta}_{i})^{2} + \sum _{i=i_{0}+2}^{m-1}
(q^{\theta}_{i-1})^{2} - c_{7}
\\
&= \|q^{\theta}\|_{2}^{2} -(q_{1}^{\theta})^{2} - (q_{m-1}^{\theta})^{2}
- (q_{m}^{\theta})^{2} -c_{7} \geq \|q^{\theta}\|_{2}^{2} - c_{8}(L)
\sqrt{m},
\end{align*}
	where we used the tail-estimate \eqref{eq:Wasserstein.tail.decay.norm.equivalence} in the last inequality.
	The corresponding upper estimate on $B$ follows from the same argument, showing that indeed 
	\[|B- \|q^\theta\|_2^2 | \leq  c_9(L)\sqrt{m}.\]
	
	Case 2 --- Assumption (A2) holds.

	Fix $1<i<m$ and observe that by a first order Taylor expansion of $F^{-1}_{\mu^\theta}$ followed by Cheeger's inequality \eqref{eq:cheeger}, we have that
	\[ |q_{i-1}^\theta-q_{i}^\theta|
	=\left| F_{\mu^\theta}^{-1} \left( \tfrac{i-1}{m} \right)  -  F_{\mu^\theta}^{-1} \left( \tfrac{i}{m} \right) \right|
	 \leq \frac{1}{h \min\{i-1, m-i\} }, \]
	where $h$ is the absolute constant appearing in \eqref{eq:cheeger}.
	In particular, $\sum_{i=2}^{m-1} (q_{i-1}^\theta-q_{i}^\theta)^2 \leq c_{10}/h^2$, as required.
\end{proof}

\begin{Lemma}
\label{lem:Gamma.theta.geq.SW2}
	Let $X$ be centred and isotropic. If either
	\begin{enumerate}
	\item[(A1)] $\sup_{\theta \in S^{d-1}} \|\inr{X,\theta}\|_{L_4} \leq L$ for some $L\geq 1$, or
	\item[(A2)] $X$ is log-concave,
	\end{enumerate}
	then there is a constant $c_1$ that depends only on $L$ and an absolute constant $c_2$ such that 
	\[   \sup_{\theta\in S^{d-1}} \left( \frac{1}{m}\sum_{i=1}^m \left( \inr{X_i,\theta}^\sharp - \lambda_i^\theta\right)^2 \right)^{1/2}
	\geq \mathcal{SW}_2(\mu_m,\mu) -
	\begin{cases}
	 c_1/ m^{1/4} & \text{if  (A1) holds},\\
	c_2 \log(m)/\sqrt{m} & \text{if  (A2) holds}.
 	 \end{cases}\]
\end{Lemma}
\begin{proof}
	Fix $\theta\in S^{d-1}$.
	For $u\in(0,1)$, set $ i(u)$  to be the smallest integer larger than $um$.
	Then $u\in(\frac{i-1}{m},\frac{i}{m}]$ and  by the definition of the right-inverse, $F^{-1}_{\mu^\theta_m}(u)=\inr{X_{i(u)},\theta}^\sharp$.

	Moreover, by the triangle inequality,
	\begin{align*}
	 \mathcal{W}_2(\mu^\theta_m,\mu^\theta)
	&= \left( \int_0^1 \left( F^{-1}_{\mu^\theta_m}(u) - F^{-1}_{\mu^\theta}(u) \right)^2\,du \right)^{1/2} \\
	&\geq \left( \int_0^1 \left( \inr{X_{i(u)},\theta}^\sharp - \lambda^\theta_{i(u)} \right)^2\,du \right)^{1/2} 
	-\left( \int_0^1 \left( \lambda^\theta_{i(u)} - F^{-1}_{\mu^\theta}(u) \right)^2\,du \right)^{1/2}
	=A+B.
	\end{align*}
	Clearly, $A^2 = \frac{1}{m}\sum_{i=1}^m( \inr{X_{i},\theta}^\sharp - \lambda^\theta_{i})^2$.
	Thus, to complete the proof it remains to show that $B\leq c_1(L)/m^{1/4}$ when (A1) holds, and  that $B\leq c_2\log(m)/\sqrt{m}$  when (A2) holds.

	The proofs of these facts follow using identical arguments to the ones presented in the proof of Lemma \ref{lem:Gamma.theta.quantils.averaged.vs.not}, and we omit the details.
\end{proof}

\section{Lower bounds on the Wasserstein distance} \label{sec:lower-W2}

This section is devoted to showing that Theorem \ref{thm:Wasserstein.small} and Theorem \ref{cor:Wasserstein.small.log-concave.subgaussian} are optimal in a strong sense.

\subsubsection*{Estimates with constant probability}
We begin with lower bounds on $\mathcal{SW}_2(\mu_m,\mu)$ that hold with constant probability, exhibiting that the dependence on $\Delta$ in Theorem \ref{thm:Wasserstein.small} and Theorem \ref{cor:Wasserstein.small.log-concave.subgaussian} cannot be improved even if one is willing to accept much weaker probability estimates. 

\begin{Lemma} \label{lemma:lower-weak-nontrivial}
Let $X$ be centred and isotropic, and set $0<\beta<1$ such that $\E\|X\|_2 \geq \beta \sqrt{d}$. 
Then with probability at least $c_0(\beta)$,
\begin{equation} \label{eq:S-lower-1}
\mathcal{SW}_2(\mu_m,\mu) \geq c_1(\beta)\sqrt{\frac{d}{m}}.
\end{equation}
\end{Lemma}

Note that if $X$ is isotropic and log-concave then $\E\|X\|_2 \geq c\sqrt{d}$ for an absolute constant $c$. In particular, the dependence on $\Delta$ in  Theorem \ref{cor:Wasserstein.small.log-concave.subgaussian} cannot be improved (up to the logarithmic factor) even in the constant probability regime.

\vspace{0.5em}
The proof of Lemma \ref{lemma:lower-weak-nontrivial} requires a standard observation that will be used again in what follows.

Let $\tau$ be a centred probability measure on $\R$ and set $Y$ to be distributed according to $\tau$. Let $(Y_i)_{i=1}^m$ be independent copies of $Y$ and put $\tau_m$ to be the corresponding empirical measure. Consider an optimal coupling $\Pi$ of $\tau_m$ and $\tau$ in the $\mathcal{W}_2$ sense. Since $\tau$ is centred, it follows from the Cauchy-Schwarz inequality that
\begin{equation} \label{eq:S-coupling-1}
\left|\frac{1}{m}\sum_{i=1}^m Y_i\right| 
= \left|\int_{\R} x \, \tau_m(dx)\right|
=\left|\int_{\R \times \R} (x-y) \,\Pi(dx,dy)\right| 
\leq  \mathcal{W}_2(\tau_m,\tau).
\end{equation}

\vspace{0.5em}
\begin{proof} [Proof of Lemma \ref{lemma:lower-weak-nontrivial}]
By \eqref{eq:S-coupling-1} applied to all the one-dimensional marginals of $X$, \begin{equation} \label{eq:S-coupling-2}
\left\|\frac{1}{m}\sum_{i=1}^m X_i \right\|_2 \leq \mathcal{SW}_2(\mu_m,\mu).
\end{equation}
In particular, to establish \eqref{eq:S-lower-1} it suffices to show that if $\E\|X\|_2 \geq \beta \sqrt{d}$, then with probability at least $c_0(\beta)$
\[
\left\|\frac{1}{m}\sum_{i=1}^m X_i \right\|_2 \geq c_1(\beta) \sqrt{\frac{d}{m}}.
\]
To that end, note that by isotropicity, $\E\|X\|_2^2=d$. Invoking the Paley-Zygmund inequality (see, e.g., \cite{de2012decoupling}), there are constants $c_2$ and $c_3$ that depend on $\beta$ for which
\[
\PP\left( \|X\|_2 \geq c_2 \sqrt{d}\right) \geq c_3.
\]
A standard binomial estimate shows that there is a constant $c_4(\beta)$ such that with probability at least $1-2\exp(-c_4 m)$,
\[
\left|\left\{ i : \|X_i\|_2 \geq c_2 \sqrt{d} \right\}\right| \geq \frac{c_3}{2}m,
\]
and by symmetrization and the Kahane-Khintchine inequality (see, e.g., \cite{ledoux1991probability}),
\begin{equation*}
\E \left\| \frac{1}{m} \sum_{i=1}^m X_i \right\|_2 \geq
c_5 \E \left( \frac{1}{m^2}\sum_{i=1}^m \|X_i\|_2^2 \right)^{1/2} \geq c_6 (\beta) \sqrt{\frac{d}{m}}.
\qedhere
\end{equation*}
\end{proof}

Next, let us establish the optimality of Theorem \ref{thm:Wasserstein.small} by constructing an isotropic random vector $X$ that satisfies $L_4-L_2$ norm equivalence and with constant probability
\begin{equation} \label{eq:S-lower-3}
\mathcal{SW}_2(\mu_m,\mu) \geq c \left(\frac{d}{m}\right)^{1/4}.
\end{equation}
In particular, setting $\Delta \sim d/m$, \eqref{eq:S-lower-3} implies that the term $\Delta^{1/4}$ in Theorem \ref{thm:Wasserstein.small} is the best one can hope for, even if one allows a probability estimate of a constant rather than of $1-\exp(-c_1\Delta m)$.

\begin{Remark}
As it happens, the random vector $X$  we construct satisfies that with constant probability,
\begin{equation} \label{eq:S-lower-2}
\sup_{\theta \in S^{d-1}} \left| \frac{1}{m}\sum_{i=1}^m \inr{X_i,\theta}^2 - 1 \right| \leq c_2 \sqrt{\frac{d}{m} \log\left(\frac{em}{d}\right)},
\end{equation}
showing that there can be a substantial gap between $\mathcal{SW}_2(\mu_m,\mu)$ and $\rho_{d,m}$.
\end{Remark}

\begin{Example} \label{ex:lower-rv}
Let $\beta>0$ to be named in what follows. Set $W$ to be a random vector distributed uniformly in $S^{d-1}$ and let $v$ be a real-valued random variable that is independent of $W$ and takes the values $\beta \sqrt{d}$ with probability $1-\frac{1}{2m}$ and $(md)^{1/4}$ with probability $\frac{1}{2m}$.

The wanted random vector is $X=vW$; it is rotation invariant, and for a well-chosen $\beta \sim 1$ it is also isotropic. Moreover, it is straightforward to verify that with that choice of $\beta$,
$\sup_{\theta \in S^{d-1}} \|\inr{X,\theta}\|_{L_4} \leq C$ for an absolute constant $C$, implying that $X$ satisfies $L_4-L_2$ norm equivalence.

Let $\theta \in S^{d-1}$ and observe that with probability at least $1-\frac{1}{2m}$, $\inr{X,\theta} \leq \beta \sqrt{d}$. 
Thus $F^{-1}_{\mu^\theta}(u) \leq \beta \sqrt{d}$ for $u<  1-\frac{1}{2m}$. On the other hand, with constant probability,
\[
\sup_{\theta \in S^{d-1}} \max_{1 \leq i \leq m} |\inr{X_i,\theta}| =   \max_{1 \leq i \leq m} \|X_i\|_2 = (md)^{1/4},
\]
and on that event there is some $\theta$ such that $F^{-1}_{\mu_m^\theta}(u) = (md)^{1/4}$ for $u > 1-\frac{1}{m}$. Therefore, by Lemma \ref{lem:Wasserstein.via.inverse.functions},
\begin{align*}
\mathcal{W}_2^2\left( \mu_m^\theta,\mu^\theta \right) 
\geq & \int_{1-1/m}^{1-1/2m} \left(F^{-1}_{\mu_m^\theta}(u)-F^{-1}_{\mu^\theta}(u)\right)^2 \,du
\\
\geq & \int_{1-1/m}^{1-1/2m} \left((md)^{1/4} - \beta \sqrt{d} \right)^2 \,du 
\geq \frac{1}{4}\left(\frac{d}{m}\right)^{1/2},
\end{align*}
provided that $m \geq c(\beta)d$. 
Thus, $X$ satisfies \eqref{eq:S-lower-3} with constant probability, as claimed.

Regarding \eqref{eq:S-lower-2}, that follows immediately from Tikhomirov's estimate from \cite{tikhomirov2018sample} on the extremal singular values of random matrices with independent rows. Indeed, $X$ satisfies $L_4-L_2$ norm equivalence and $\|X\|_2 \leq (md)^{1/4}$ almost surely.
\qed
\end{Example}

The random vector $X$ constructed in Example \ref{ex:lower-rv} might be considered `a-typical'. 
Next we turn to an example which shows that even for ``nice" random vectors, the best that one can hope for in Theorem \ref{thm:Wasserstein.small} is an upper bound of order $\Delta^{1/4}$ that holds with probability at least $1-\exp(-c\Delta m)$.

\subsubsection*{The Bernoulli vector}

Let $X$ be distributed uniformly in $\{-1,1\}^d$. In particular, $X$ is isotropic and $L$-subgaussian for an absolute constant $L$. It is standard to verify that for subgaussian random vectors, $\rho_{d,m} \leq c_0 \sqrt{d/m}$ with probability at least $1-2\exp(-c_1 m)$, and hence, Theorem \ref{thm:Wasserstein.small} implies that for such a random vector, with probability at least $1-\exp(-c_2\Delta m)$,
\[
\mathcal{SW}_2(\mu_m,\mu) \leq c_3 \Delta^{1/4}.
\]
Here, we show that for the uniform distribution in $\{-1,1\}^d$,
\[
\mathcal{SW}_2(\mu_m,\mu) \geq c_4 \Delta^{1/4}
\]
with probability at least $c_5\exp(-c_6\Delta m)$.

As it happens, the reason for this lower bound  is actually one dimensional.
It is caused by the behaviour of
\[
\nu = \frac{1}{2}(\delta_{-1}+\delta_1),
\]
corresponding to a marginal of $X$ in a coordinate direction.

\begin{Lemma} \label{lem:Delta.1.4}
There are absolute constants $c_1,c_2$ and $c_3$ such that, for $\Delta<1/4$ and $m\geq 4$, with probability at least $c_1\exp(-c_2\Delta m)$,
	\[  \mathcal{W}_2(\nu_m,\nu) \geq c_3 \Delta^{1/4}.\]
\end{Lemma}

Lemma \ref{lem:Delta.1.4} can be established via a standard lower bound on the binomial distribution and Lemma \ref{lem:Wasserstein.via.inverse.functions}.
For the sake of simplicity the argument we present is based on an well-known estimate due to Montgomery-Smith \cite{montgomery1990distribution}. To formulate that fact, denote by $x^\ast$  the monotone non-increasing rearrangement of $(|x_i|)_{i=1}^m$.

\begin{Lemma} \label{lem:MS}
	There are absolute constants $c_1,c_2$ and $c_3$ such that the following holds.
	Let $(\eps_i)_{i=1}^m$ be distributed as $\nu^{\otimes m}$, consider $x\in\R^m$ and set $0< u\leq m-2$.
	Then, with probability at least $c_1\exp(-c_2u)$,
	\begin{align}
	\label{eq:M.S}
	\sum_{i=1}^m  x_i \eps_i \geq
	c_3 \left( \sum_{i=1}^{\lfloor  u\rfloor}  x^\ast_i + \sqrt{u}  \left( \sum_{i= \lfloor u\rfloor + 1 }^m  (x^\ast_i)^2 \right)^{1/2}\right).
	\end{align}
\end{Lemma}

\begin{proof}[Proof of Lemma \ref{lem:Delta.1.4}]
Let $\nu_m=\frac{1}{m}\sum_{i=1}^m\delta_{\eps_i}$ and note that for $x=(\frac{1}{m},...,\frac{1}{m})$ and $u=c_0\Delta m$ for an absolute constant $c_0$,  the RHS in \eqref{eq:M.S} is at least $c_1\sqrt \Delta$. Moreover, when \eqref{eq:M.S} holds,
	\[
F_{\nu_m}(-1) \leq \tfrac{1}{2}-c_2\sqrt\Delta.
\]
Hence, $F_{\nu_m}^{-1}(u)=1$ for $u\in (\frac{1}{2}-c_2\sqrt\Delta, 1]$, but $F_\nu^{-1}(u)=-1$ for $u\in[0,\frac{1}{2}]$.
Therefore,
		\[ \int_{1/2-c_2\sqrt\Delta }^{1/2}  \left( F_{\nu_m}^{-1}(u) - F_{\nu}^{-1}(u)\right)^2 \,du
	= c_2\sqrt\Delta
\]
	and the claim follows from Lemma \ref{lem:Wasserstein.via.inverse.functions}.
\end{proof}

\begin{Remark}
	By combining \eqref{eq:S-coupling-1} and Lemma \ref{lem:MS}, it is straightforward to show that even  `regular' random vectors satisfy $\mathcal{SW}_2(\mu_m,\mu)\geq c_1\sqrt\Delta$  with probability at least $c_2\exp(-c_3\Delta m)$.
	In particular, the estimate of $\sim\Delta^{1/2}$ is the best one can hope for even when $m\gg d$.
\end{Remark}

\section{Variations on a theme}
\label{sec:Variations}

\subsubsection*{The first order max-sliced Wasserstein distance}

	Recall that  the first order Wasserstein distance is defined on $\mathcal{P}_1(\R^d)$---the Borel probability measures on $\R^d$ with finite first moment---by
	\[\mathcal{W}_1(\nu,\tau)
	=\inf_{\Pi \text{ coupling}} \int_{\R^d\times \R^d} \|x-y\|_2\,\Pi(dx,dy).
\]
	In particular, by H\"older's inequality, $\mathcal{W}_1\leq\mathcal{W}_2$.
	And, just like $\mathcal{SW}_2$, one can define the first order max-sliced	Wasserstein distance by
	$\mathcal{SW}_1(\mu,\nu)=\sup_{\theta\in S^{d-1}} \mathcal{W}_1(\mu^\theta,\nu^\theta)$.
	
	\begin{Proposition}	
	\label{cor:W1}

	 Let $X$ be centred and isotropic, and assume that $\sup_{\theta \in S^{d-1}} \|\inr{X,\theta}\|_{L_q} \leq L$ for some $q>4$.
	Then there are constants $c_0$ and $c_1$ that depend on $q$ and $L$ and absolute constants $c_2,c_3$ such that the following holds.

	Let  $0<\Delta\leq c_0$ and set $m \geq c_2 \frac{d}{\Delta}$.
	With probability at least $1-\exp(-c_3 \Delta m)$,
	\[ 
	\mathcal{SW}_1\left( \mu_m , \mu\right) 
	\leq  c_1 \left( 	\rho_{d,m } +  \sqrt{ \Delta} \right).
	\]
	\end{Proposition}
	
	The dependence on $\Delta$ in  Proposition \ref{cor:W1} is  optimal, as can be seen using the arguments presented in Section \ref{sec:lower-W2}.
	In particular, even if $X$ is log-concave,  $\mathcal{SW}_1(\mu_m,\mu)\geq c_1\sqrt{d/m}$ with constant probability and $\mathcal{SW}_1(\mu_m,\mu)\geq c_4\sqrt\Delta$ with probability at least $c_2\exp(-c_3\Delta m)$.

	\begin{proof}[Proof of Proposition \ref{cor:W1} (sketch)]
	The proof follows along the same lines as the proof for $\mathcal{SW}_2$---making use of the representation $\mathcal{W}_1(\mu^\theta_m,\mu^\theta)=\int_0^1 |F^{-1}_{\mu^\theta_m}(u)-F^{-1}_{\mu^\theta}(u)|\,du$.
	A minor modification is needed when estimating 
	\[(1)=\int_{[0,1]\setminus(\delta,1-\delta)} |F^{-1}_{\mu^\theta_m}(u)|\,du. \]
	To that end, note that by the Paley–Zygmund inequality  there are constants $c_1$ and $c_2$ depending only on $L$ such that for every $\theta\in S^{d-1}$, $F^{-1}_{\mu^\theta}(c_1)\leq -c_2$ and $F^{-1}_{\mu^\theta}(1-c_1)\geq c_2$.
	Fix $\delta\leq c_1/2$ and consider a realization of $(X_i)_{i=1}^m$ for which \eqref{eq:ratio.estiamte} holds.
	Lemma \ref{lem:psi.pm} and Lemma \ref{lem:inverse.estimate} imply that  for $u\in[0,\delta]$,
	\[ F^{-1}_{\mu^\theta_m}(u)
	\leq F^{-1}_{\mu^\theta_m}(\delta)
	\leq F^{-1}_{\mu^\theta}(2\delta)
	\leq F^{-1}_{\mu^\theta}(c_1)
	\leq -c_2;\]
	and using identical arguments, for $u\in[1-\delta,1]$, $F^{-1}_{\mu^\theta_m}(u)\geq c_2$.
	Hence
	\[(1) \leq \frac{1}{c_2}\int_{[0,1]\setminus(\delta,1-\delta)} |F^{-1}_{\mu^\theta_m}(u)|^2\,du\]
	and the estimate on $(1)$  follows from  Lemma \ref{lem:tail.integral.bounded.by.Bai.Yin.plus.ratio}.
	\end{proof}
	
	\begin{Remark}
	It is likely that the assumption $q>4$ in Proposition \ref{cor:W1} can be relaxed, in which case $\rho_{d,m}$ should be replaced by (an estimate on) $\sup_{\theta\in S^{d-1}} \frac{1}{m} \sum_{i=1}^{\delta m} |\inr{X_i,\theta}|^\ast$.
	\end{Remark}
	
	We cut further discussions on $\mathcal{SW}_1$ short because in the context of this note,  controlling $\mathcal{SW}_1(\mu_m,\mu)$ is significantly less interesting than controlling  $\mathcal{SW}_2(\mu_m,\mu)$.
	The reason for that is the Kantorovich-Rubinstein duality:
	\[\mathcal{SW}_1(\nu,\mu)
	=\sup_{\Phi\colon\R\to\R \text{ is  $1$-Lipschitz}} \left( \int_{\R^d} \Phi(\inr{\,\cdot\, ,\theta})\,d\nu - \int_{\R^d} \Phi(\inr{\,\cdot\, ,\theta}) \,d\mu\right).\]
	Therefore, $\mathcal{SW}_1$ is the difference of integrals rather than the integral of differences (like in the case of $\mathcal{SW}_2)$, making it significantly easier to control.	
	
	In the next section we present applications of a similar flavour,  where the key is again controlling differences of certain integrals.

\subsubsection*{Uniform estimation of increasing functions}

There are many diverse applications where one is concerned with estimating $\E\Phi(\inr{X,\theta})$ uniformly over $\theta\in S^{d-1}$ for certain choices of $\Phi$ and the given data is an independent sample $(X_i)_{i=1}^m$.
For instance, the choice $\Phi(x)=x^2$ corresponds to covariance estimation where the natural (yet surprisingly suboptimal) idea was to use the empirical quadratic mean as an estimator and bound 
\begin{align}
\label{eq:cov.estimate}
\sup_{\theta\in S^{d-1}} \left| \frac{1}{m}\sum_{i=1}^m \inr{X_i,\theta}^2 - \E\inr{X,\theta}^2 \right|.
\end{align}
When $X$ is  isotropic, \eqref{eq:cov.estimate} is simply $\rho_{d,m}$.
And, as noted previously, the question of controlling $\rho_{d,m}$ has been studied extensively under various conditions on the isotropic vector $X$.

Another natural choice is $\Phi(x)=|x|^p$, see e.g., \cite{guedon2007lp} in the log-concave case, where again, the empirical $p$-mean was the chosen estimator and $\sup_{\theta\in S^{d-1}} | \frac{1}{m}\sum_{i=1}^m |\langle X_i,\theta\rangle|^p - \E |\langle X,\theta\rangle|^p |$ was studied.

Unfortunately, the empirical mean has a major drawback (particularly in high dimensions): the  error  deteriorates dramatically when $X$ has tails that are heavier than  a gaussian---see \cite{lugosi2019mean} for a survey and recent developments in the direction.
For example, the main result in \cite{mendelson2021approximating} focuses on the case where $\Phi(x)=|x|^p$, and it is shown that a modified estimator
 \[\widehat{\mathcal{E}}(\Phi,\theta,\delta,(X_i)_{i=1}^m)=\frac{1}{m}\sum_{i=\lfloor \delta m+1\rfloor}^{\lfloor m-\delta m \rfloor} \Phi\left( \inr{X_i,\theta}^\sharp\right)\]
almost recovers the best possible statistical error once the parameter $\delta$ is suitably chosen.
The proof in \cite{mendelson2021approximating} relies heavily on the fact that $\Phi(x)=|x|^p$, but as we show here, this actually has nothing to do with the particular structure of $\Phi$.
As it happens, the optimal estimate holds uniformly over monotone functions that merely satisfy a compatible growth condition.
Indeed, let $p\in[1,\infty)$,  $C\geq 1$, and set 
\[ \mathcal{I}_{p,C}=\left\{ \Phi\colon\mathbb{R}\to\mathbb{R} : \Phi \text{ non-decreasing and }\sup_{x\in\R} \frac{|\Phi(x)|}{1 + |x|^p} \leq C \right\}.\]

\begin{Proposition}
\label{prop:estimate.increasing.functions}
	Let $X$ be centred and isotropic and assume that $\sup_{\theta \in S^{d-1}} \|\inr{X,\theta}\|_{L_q} \leq L$ for some $q\geq 2p$.
	Then there are absolute constants $c_0,c_1,c_2,c_3$ and a constant $c_4$ that depends only on $p$, $q$ and $L$ such that the following holds.

	Let  $0<\Delta\leq c_0$, set $m \geq c_1 \frac{d}{\Delta}$, and consider $\delta=c_2\Delta\log^2(\frac{e}{\Delta})$.
	With probability at least $1-\exp(-c_3 \Delta m)$,
	for all $\Phi\in \mathcal{I}_{p,C}$ and $\theta\in S^{d-1}$, 
	\begin{align*}
	\left| \widehat{\mathcal{E}}(\Phi,\theta,\delta,(X_i)_{i=1}^m)- \E\Phi(\inr{X,\theta} ) \right|
	\leq c_4 C
	\begin{cases}
	 \sqrt{ \Delta}   &\text{if } q>2p,\\
	\sqrt\Delta \log^2\left(\frac{1}{\Delta}\right)  &\text{if } q=2p.
	\end{cases} 
	\end{align*}
\end{Proposition}

\begin{Remark}
	Note that on the `good event' of Proposition \ref{prop:estimate.increasing.functions}, one may use \emph{any} function $\Phi\in \mathcal{I}_{p,C}$.
\end{Remark}

	In a similar fashion to the argument used in Section \ref{sec:lower-W2} one can show that the requirement that $m\geq c\frac{d}{\Delta}$ is necessary.
	Moreover,  even for  $\Phi(x)=x$ and a fixed $\theta\in S^{d-1}$,  the error in Proposition \ref{prop:estimate.increasing.functions} is the best that one can hope for in the following sense: there are suitable constants $c_1,c_2,c_3$ such that for every (fixed) estimator\footnote{In this context, an estimator is a functional $\widehat{\mathcal{E}}'\colon [0,1]\times (\R^d)^m\to\R$ that receives as input the sample $(X_i)_{i=1}^m$ and the wanted accuracy $\Delta$.},  $|\widehat{\mathcal{E}}'(\Delta,(X_i)_{i=1}^m) -\E\inr{X,\theta}|\geq c_1\sqrt{\Delta}$ with probability at least $ c_2\exp(-c_3\Delta m)$ (see \cite{catoni2012challenging} for a precise statement).
	In particular, Proposition \ref{prop:estimate.increasing.functions} is optimal for $q>2p$.

	\begin{Remark}
	The assertion of Proposition \ref{prop:estimate.increasing.functions} remains valid when extending $\mathcal{I}_{p,C}$---by adding functions of the form $\Phi(|\cdot|)$ for $\Phi\in \mathcal{I}_{p,C}$, i.e.\ replacing  $\mathcal{I}_{p,C}$ by
	\[\mathcal{I}_{p,C}'=\mathcal{I}_{p,C}\ \cup \{ \Phi(|\cdot|) : \Phi\in \mathcal{I}_{p,C}\};\]
	the modifications needed in the proof are minor and are omitted.
	
	Applied to $\mathcal{I}_{p,C}'$, Proposition \ref{prop:estimate.increasing.functions} is an improvement of  the main result in \cite{mendelson2021approximating}.
	The latter  has the restriction $\Delta \geq c_1' \frac{d}{m}\log(\frac{em}{d})$ rather than  $\Delta \geq c_1' \frac{d}{m}$.
	\end{Remark}

\begin{proof}[Proof of Proposition \ref{prop:estimate.increasing.functions}]
	We may and do assume without loss of generality that $\delta m$ is integer.
	Fix a realization $(X_i)_{i=1}^m$ that satisfies \eqref{eq:ratio.estiamte}, denote by $\kappa$  the absolute constant in Lemma \ref{lem:psi.pm} and set $\delta=\kappa\Delta \log^2(\frac{e}{\Delta})$.
	Lemma \ref{lem:inverse.estimate} and the monotonicity of $\Phi$ imply that
	\begin{align*}
		\widehat{\mathcal{E}}(\Phi,\theta,\delta,(X_i)_{i=1}^m)
		&=\int_{\delta}^{1-\delta} \Phi\left( F^{-1}_{\mu^\theta_m}(u) \right) \,du
		\leq \int_{\delta}^{1-\delta} \Phi\left( F^{-1}_{\mu^\theta}(\psi_+(u)) \right) \,du.	
	\end{align*}
	Moreover, by a change of variable $u\leftrightarrow\psi_+(u)$, it is evident that
	\begin{align*}
	&\int_{\delta}^{1-\delta} \Phi\left( F^{-1}_{\mu^\theta}(\psi_+(u)) \right) \,du \\
	&=\int_{\psi_+(\delta)}^{\psi_+(1-\delta)} \Phi\left( F^{-1}_{\mu^\theta}(v) \right) \,dv  +\int_{\delta}^{1-\delta} \Phi\left( F^{-1}_{\mu^\theta}(\psi_+(u)) \right)(1-\psi'_+(u)) \,du
	=(1) + (2).
	\end{align*}
	Next, it follows from the growth assumption on $\Phi$ and the estimate on $F_{\mu^\theta}$ from the tail decay of the marginal (see \eqref{eq:Wasserstein.tail.decay.norm.equivalence})  that there is a constant $c_1=c_1(p,q,L)$ such that  $|\Phi(F^{-1}_{\mu^\theta}(u))|\leq c_1 Cu^{-p/q}$ for all $u\in(0,1)$.
	Setting $U=[0,1]\setminus(\psi_+(\delta),\psi_+(1-\delta))$, there is a constant $c_2=c_2(p,q,L)$ such that
	\[ (1)- \E\Phi(\inr{X,\theta})
	=\int_U \Phi\left( F^{-1}_{\mu^\theta}(v) \right)\,dv
	\leq  c_2 C 
	\begin{cases}
	\sqrt{\Delta} &\text{if } q>2p, \\
	\sqrt{\Delta}\log(\frac{1}{\Delta}) &\text{if } q=2p.
	\end{cases}
	\]
	In a similar fashion, it is straightforward  to verify that there is a constant $c_3=c_3(p,q,L)$ such that 
	\[(2) \leq c_3 C 
	\begin{cases}
	\sqrt{\Delta} &\text{if } q>2p, \\
	\sqrt{\Delta}\log^2(\frac{1}{\Delta}) &\text{if } q=2p.
	\end{cases}\]
	Hence,
	\[ \widehat{\mathcal{E}}(\Phi,\theta,\delta,(X_i)_{i=1}^m)-\E\Phi(\inr{X,\theta})
	\leq  (c_2+c_3) C
	\begin{cases}
	\sqrt{\Delta} &\text{if } q>2p, \\
	\sqrt{\Delta}\log^2(\frac{1}{\Delta}) &\text{if } q=2p.
	\end{cases}\]
	The proof of the corresponding lower bound follows from an identical argument and is omitted.
\end{proof}

\vspace{1em}
\noindent
{\bf Acknowledgements:}
This research was funded in whole or in part by the
Austrian Science Fund (FWF) [doi: 10.55776/P34743 and 10.55776/ESP31] and the Austrian National Bank [Jubil\"aumsfond, project 18983].
The authors would like to thank the referee for suggestions that helped to improve the paper.

\bibliographystyle{abbrv}

\begin{thebibliography}{10}

\bibitem{adamczak2010quantitative}
R.~Adamczak, A.~Litvak, A.~Pajor, and N.~Tomczak-Jaegermann.
\newblock Quantitative estimates of the convergence of the empirical covariance
  matrix in log-concave ensembles.
\newblock {\em Journal of the American Mathematical Society}, 23(2):535--561,
  2010.

\bibitem{artstein2015asymptotic}
S.~Artstein-Avidan, A.~Giannopoulos, and V.~D.~Milman.
\newblock {\em Asymptotic {G}eometric {A}nalysis, {P}art I}, volume 202.
\newblock American Mathematical Society, 2015.

\bibitem{bai1993limit}
Z.-D.~Bai and Y.-Q.~Yin.
\newblock Limit of the smallest eigenvalue of a large dimensional sample
  covariance matrix.
\newblock {\em The Annals of Probability}, 21(3):1275--1294. 1993.

\bibitem{bartl2022random}
D.~Bartl and S.~Mendelson.
\newblock Random embeddings with an almost {G}aussian distortion.
\newblock {\em Advances in Mathematics}, 400:108261, 2022.

\bibitem{bartl2022nongaussian}
D.~Bartl and S.~Mendelson.
\newblock   Optimal non-gaussian Dvoretzky-Milman embeddings.
\newblock {\em International Mathematics Research Notices}, (10):8459--8480, 2024.

\bibitem{bartl2024empirical}
D.~Bartl and S.~Mendelson.
\newblock Empirical approximation of the gaussian distribution in $\mathbb{R}^d$.
\newblock {\em Advances in Mathematics}, to appear.

\bibitem{bobkov1996extremal}
S.~Bobkov.
\newblock Extremal properties of half-spaces for log-concave distributions.
\newblock {\em The  Annals of Probability}, 24(1):35--48, 1996.

\bibitem{bobkov2019one}
S.~Bobkov and M.~Ledoux.
\newblock {\em One-dimensional empirical measures, order statistics, and
  {K}antorovich transport distances}, volume 261.
\newblock American Mathematical Society, 2019.

\bibitem{bobkov1999isoperimetric}
S.~Bobkov.
\newblock Isoperimetric and analytic inequalities for log-concave probability
  measures.
\newblock {\em The  Annals of Probability}, 27(4):1903--1921, 1999.

\bibitem{boedihardjo2024sharp}
M.~Boedihardjo.
\newblock Sharp bounds for max-sliced Wasserstein distances.
\newblock {\em preprint}, 2024.

\bibitem{boucheron2013concentration}
S.~Boucheron, G.~Lugosi, and P.~Massart.
\newblock {\em Concentration inequalities: {A} nonasymptotic theory of
  independence}.
\newblock Oxford university press, 2013.

\bibitem{catoni2012challenging}
O.~Catoni.
\newblock Challenging the empirical mean and empirical variance: a deviation study.
\newblock {\em Annales de l’Institut Henri Poincar\'e, Probabilites et Statistiques}, 48(4):1148–1185, 2012.

\bibitem{de2012decoupling}
V.~De~la Pena and E.~Gin{\'e}.
\newblock {\em Decoupling: from dependence to independence}.
\newblock Springer Science \& Business Media, 2012.

\bibitem{deshpande2019max}
I.~Deshpande, Y.-T.~Hu, R.~Sun, A.~Pyrros, N.~Siddiqui, S.~Koyejo, Z.~Zhao,
  D.~Forsyth, and A.~G. Schwing.
\newblock Max-sliced {W}asserstein distance and its use for gans.
\newblock In {\em Proceedings of the IEEE/CVF Conference on Computer Vision and
  Pattern Recognition}, pages 10648--10656, 2019.

\bibitem{dudley1969speed}
R.~M.~Dudley.
\newblock The speed of mean {G}livenko-{C}antelli convergence.
\newblock {\em The Annals of Mathematical Statistics}, 40(1):40--50, 1969.


\bibitem{figalli2021invitation}
A.~Figalli and F.~Glaudo.
\newblock {\em An {I}nvitation to {O}ptimal {T}ransport, {W}asserstein
  {D}istances, and {G}radient {F}lows}.
\newblock EMS Textbooks in Mathematics, 2021.

\bibitem{guedon2007lp}
O.~Gu{\'e}don and M.~Rudelson.
\newblock { $L_p$-moments of random vectors via majorizing measures}.
\newblock {\em Advances in Mathematics}, 208(2), 798--823, 2007.




\bibitem{kannan1995isoperimetric}
R.~Kannan, L.~Lov{\'a}sz, and M.~Simonovits.
\newblock Isoperimetric problems for convex bodies and a localization lemma.
\newblock {\em Discrete \& Computational Geometry}, 13(3):541--559, 1995.


\bibitem{ledoux2001concentration}
M.~Ledoux.
\newblock {\em The concentration of measure phenomenon}.
\newblock Number~89. American Mathematical Soc., 2001.

\bibitem{ledoux1991probability}
M.~Ledoux and M.~Talagrand.
\newblock {\em Probability in {B}anach {S}paces: isoperimetry and processes},
  volume~23.
\newblock Springer Science \& Business Media, 1991.

\bibitem{lin2021projection}
T.~Lin, Z.~Zheng, E.~Chen, M.~Cuturi, and M.~I. Jordan.
\newblock On projection robust optimal transport: {S}ample complexity and model
  misspecification.
\newblock In {\em International Conference on Artificial Intelligence and
  Statistics}, pages 262--270. PMLR, 2021.

\bibitem{lugosi2020multivariate}
G.~Lugosi and S.~Mendelson.
\newblock Multivariate mean estimation with direction-dependent accuracy.
\newblock {\em Journal of the European Mathematical Society}, to appear, 2020.

\bibitem{lugosi2019mean}
G.~Lugosi and S.~Mendelson.
\newblock {Mean estimation and regression under heavy-tailed distributions: A survey},
\newblock {\em Foundations of Computational Mathematics}, 19(5), 1145-1190, 2019.

\bibitem{manole2019minimax}
T.~Manole, S.~Balakrishnan, and L.~Wasserman.
\newblock Minimax confidence intervals for the sliced Wasserstein distance.
\newblock {\em Electronic  Journal of Statistics}, 16(1): 2252--2345, 2022.

\bibitem{mendelson2020approximating}
S.~Mendelson.
\newblock Approximating the covariance ellipsoid.
\newblock {\em Communications in Contemporary Mathematics}, 22(08):1950089,
  2020.

\bibitem{mendelson2021approximating}
S.~Mendelson.
\newblock Approximating $L_p$ unit balls via random sampling.
\newblock {\em Advances in Mathematics}, 386:107829, 2021.


\bibitem{mendelson2014singular}
S.~Mendelson and G.~Paouris.
\newblock On the singular values of random matrices.
\newblock {\em Journal of the European Mathematical Society}, 16(4):823--834,
  2014.


\bibitem{mendelson2020robust}
S.~Mendelson and N.~Zhivotovskiy.
\newblock Robust covariance estimation under $L_4-L_2$ norm equivalence.
\newblock {\em The Annals of Statistics}, 48(3):1648--1664, 2020.


\bibitem{montgomery1990distribution}
S.~J. Montgomery-Smith.
\newblock The distribution of {R}ademacher sums.
\newblock {\em Proceedings of the American Mathematical Society},
  109(2):517--522, 1990.


\bibitem{nietert2022statistical}
S.~Nietert,  R.~Sadhu, Z.~Goldfeld, and K.~Kato.
\newblock Statistical, robustness, and computational guarantees for sliced wasserstein distances,
\newblock {\em NeurIPS}, to appear, 2022.

\bibitem{olea2022generalization}
J.~Olea, C.~Rush, A.~Velez, and J.~Wiesel.
\newblock The out-of-sample prediction error of the square-root-LASSO and related estimators.
\newblock {\em arXiv preprint arXiv:2211.07608}, 2022.

\bibitem{paouris2006concentration}
G.~Paouris.
\newblock Concentration of mass on convex bodies.
\newblock {\em Geometric \& Functional Analysis GAFA}, 16(5):1021--1049, 2006.

\bibitem{paty2019subspace}
F.-P. Paty and M.~Cuturi.
\newblock Subspace robust {W}asserstein distances.
\newblock In {\em International conference on machine learning}, pages
  5072--5081. PMLR, 2019.



\bibitem{ruschendorf1985wasserstein}
L.~R{\"u}schendorf.
\newblock The {W}asserstein distance and approximation theorems.
\newblock {\em Probability Theory and Related Fields}, 70(1):117--129, 1985.


\bibitem{talagrand1994sharper}
M.~Talagrand.
\newblock Sharper bounds for {G}aussian and empirical processes.
\newblock {\em The Annals of Probability}, 22(1):28--76, 1994.


\bibitem{talagrand2022upper}
M.~Talagrand.
\newblock {\em Upper and lower bounds for stochastic processes: decomposition theorems}, volume~60.
\newblock Springer Nature, 2022.


\bibitem{tikhomirov2018sample}
K.~Tikhomirov.
\newblock Sample covariance matrices of heavy-tailed distributions.
\newblock {\em International Mathematics Research Notices},
  2018(20):6254--6289, 2018.

\bibitem{vaart1996weak}
A.~W. Vaart and J.~A. Wellner.
\newblock {\em Weak convergence and empirical processes}.
\newblock Springer, 1996.

\bibitem{villani2021topics}
C.~Villani.
\newblock {\em Topics in optimal transportation}, volume~58.
\newblock American Mathematical Soc., 2021.

\end{thebibliography}

\end{document}